\documentclass[11pt,letterpaper]{article}

\usepackage{Xigecmds}

\numberwithin{equation}{section}

\newtheorem{corollary}{Corollary}

\newtheorem{physical conclusion}{Physical Conclusion}

\newtheorem{remark}{Remark}[section]

\title{Analyzing Turing's Systems via Dynamic Bifurcation Theory}
\author{Xige Yang and Dapeng Li
}

\begin{document}
\date{}
\maketitle
\begin{abstract}
\footnotesize{In this paper, we introduce a novel approach to study reaction diffusion systems -- dynamic transition theory approach developed in \cite{Ma2015}. This approach generalizes Turing's classical result (linear stability analysis) on pattern formation and cast some new insights into Turing's systems. Specifically, we studied the Turing's instability and dynamic transition phenomenon for a Turing's system, and expressions of the critical parameters $L(D_{u},D_{v})$,and $D_{v}^{*}$ are derived. These two simple parameters are sufficient to provide us enough information on the Turing's instability result as well as the dynamic transition behavior of the system. As an application, based on the method we establish in this paper, we found that the Schnakenberg system has two different transition types : single real eigenvalue transition and double real eigenvalues transition. These transition types are interpreted using phase diagrams.}
\end{abstract}
\tableofcontents

\section{Introduction}\label{sec:intro}
Patterns are universal phenomena in physics, chemistry, biology, geography, economics, and even sociology. In biology, there had been lots studies on pattern formation since Alan Turing published his celebrated paper \textit{The Chemical Basis of Morphogenesis}, which put forth a mathematical model for spatial pattern formation. Since then, the mathematical model, now known as a reaction-diffusion system, states a stable equilibrium solution without diffusion may become unstable because of diffusion.

One of the most popular methods to understand biological pattern formation is by using reaction diffusion systems, first introduced by Alan Turing in 1952 \cite{Turing1952}. In his paper, Turing introduced the concept of "pre-pattern" as a precursor of the real pattern we observe, and hypothesized that these spatial pre-patterns are generated by biochemicals, which he called morphogen. He did lots of experiments on his own, mixing morphogens in a well-stirred system, and found that the uniform steady state is stable to small perturbations. Moreover, he showed that two different morphogens in the same system can produce unstable patterns, now known as diffusion-driven instability (DDI). That is, under suitable choices of parameters, the homogenous steady state of the system will lose its stability.

Turing's work had a great influence on interdisciplinary subjects such as mathematical biology, biophysics, non-equilibrium physical chemistry and complexity science. For example, in 1960s, IIya Prigogine and his collaborators followed up Turing's work, formulated and analyzed a model for the Belousov-Zhabotinsky reaction, which is found in the 1950s an now known as a classical example of a self-organizing chemical reaction \cite{Winfree1984}. Fig. \ref{fig:BZ_exp} is an chemical experiment of Belousov-Zhabotinsky reaction, and Fig. \ref{fig:BZ} is a numerical simulation of BZ reaction using at different times (pictures generated by \textit{MATLAB}). In Prigogine's work, they proposed a mathematical model called \textit{Brusselator model} to explain how Turing's Pattern is generated, and their model also explained why the spontaneous creation of order (or spontaneous symmetry breaking) is not forbidden by the Second Law of Thermodynamics \cite{Prigogine1985, Pechenkin2009}. This work This work earned him the 1977 Nobel Prize in Chemistry.

 Fig. \ref{fig:BZ_exp} is an chemical experiment of Belousov-Zhabotinsky reaction, and Fig. \ref{fig:BZ} is a numerical simulation of BZ reaction using at different times (pictures generated by \textit{MATLAB}).
 \begin{figure}[!ht]
 	\centering
 	\includegraphics[width = .8\textwidth]{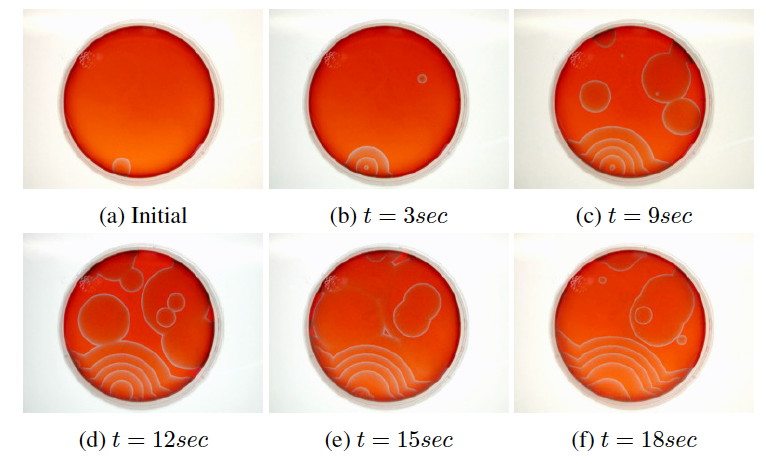}
 	\caption{An experiment of BZ reaction. Rotational chemical wave is observed in seconds, indicating chemical oscillation. Details of the experiment can be found at \href{http://www.pojman.com/NLCD-movies/NLCD-movies.html}{http://www.pojman.com/NLCD-movies/NLCD-movies.html}.}
	\label{fig:BZ_exp}
 \end{figure}

 \begin{figure}[!ht]
 	\centering
	\includegraphics[width = .8\textwidth]{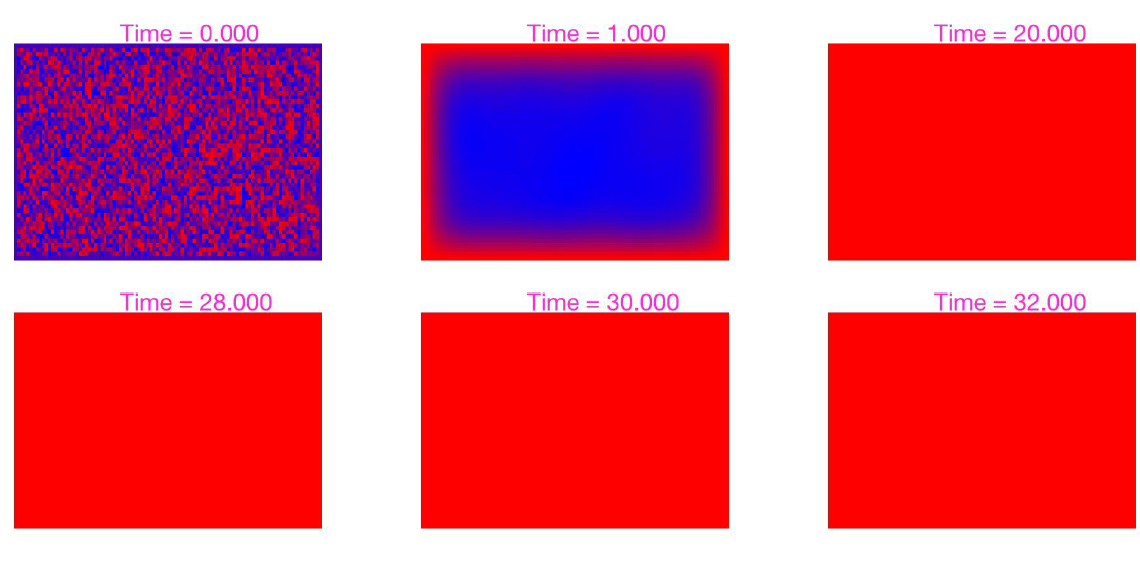}
 	\caption{A numerical simulation of BZ reaction using the Brusselator model at different time slice (time is nondimensionalized). Under a suitable set of parameters, rotation chemical waves are reproduced. The simulation is achieved combining method of line (MOL) with central difference spatial scheme.}
	\label{fig:BZ}
 \end{figure}
 
 Latter developments of Belousov-Zhabotinsky reaction are more or less extensions of Prigogine's pioneering works. For example, in Fields et. al developed a model involving 11 reactions and 12 species \cite{Field1974}, and can be further reduced in a equation system of three species \cite{Ma2011a, Ma2015}. In 1972, Alfred Gierer and Hans Meinhardt introduced the concept of activator-inhibitor system, a special kind of reaction diffusion system \cite{Gierer1972, Meinhardt1981}.  After linear analysis of the Brusselator model, they noticed that only activator-inhibitor type system may generate Turing's instability and pattern formation. Such activator-inhibitor system with the Neumann boundary condition is given and analysised in \cite{shi2009}, where $u$ is the activator satisfying $f'_{u} >0$, and $v$ is the inhibitor satisfying  $g'_{v} <0$. Besides, we also refer the readers to \cite{Maini2000}, which provided us with necessary conditions for instability of a higher-dimensional Turing's system.

Other researching directions on Turing's pattern consider more on agent-based models as well as stochastic mechanism. For example, \cite{Woolley2011b} partitioned the whole domain into small subdomains, and built up a master equation system based on these subdomains to interpret how domain size can affect the behavior of Turing's pattern.\cite{Woolley2011a} studied the robustness property of the pattern on each subdomain by applying spectral method on external noise. Results of both papers are validated by numerical methods, and summarized in a more general review article \cite{Maini2012}. 

On the other hand, the dynamic transition theory developed in \cite{Ma2015}, which is established by Ma and Wang, is a powerful mathematical tool to study the nonlinear dissipative system. Based on the dynamical phase transition theory, Ma and Wang have studied more than twenty kinds of phase transition phenomena, including the Taylor's instability, Rayleigh-B\'enard convection in the fluid dynamics, and ENSO phenomenon in atmospheric circulation. Based on the theory, we study the phase transition phenomena of the Schnakenberg system.

This paper is organized as follows. In Section \ref{sec:2}, we review some previous results and declare our main works. In Section  \ref{sec:3}, the dynamic transition theory will be introduced. In Section \ref{sec:turing}, the general method of analyzing Turing's systems will be addressed, including the necessity and sufficiency condition for Turing's instability, the method to derive the critical parameter, the classification of dynamic transitions, and center manifold reduction result. In the Section \ref{sec:5}, based on the method established in Section \ref{sec:turing}, we derive the critical parameter $D_{v}^{*}$ of the Schnakenberg system. Due to the principle of exchanges of stabilities of this system, it tells us in a clear way that under which condition the Schnakenberg system can generate the Turing's instability, which types of dynamic transitions the system may contain, and the physical meanings of the dynamic transitions.

%%%%%%%%%%%%%%%%%%%%%%%%%%%%%%%%%%%%%%%%%%%%%%%%%%%%%%%%
%%%%%%%%%%%%%%%%%%%%%%%%%%%%%%%%%%%%%%%%%%%%%%%%%%%%%%%%

\section{Statement of the problem and main results}\label{sec:2}
\subsection{Turing's instability}\label{sec:Turing}
The kind of instability driven by diffusion, which can be widely found in the activator-inhibitor systems, is so-called Turing's instability. Turing firstly noticed that such activator-inhibitor system can generate a stationary pattern, if the inhibitor diffuses faster than the activator. For over a half century, Turing's instability and pattern formation has been studied widely. In this paper, we study the following system
\bea
\begin{cases}
   \pd{u}{t}=D_{u}\Delta u + f(u,v),\\[10pt]
   \pd{v}{t}=D_{v}\Delta  v + g(u,v).
\end{cases}\label{eq:model}
\eea
where $u$ is the concentration of a short-range autocatalytic substance (i.e. activator), and $v$ is the long-range antagonist (i.e. inhibitor) of $u$.

 We obtain the necessary and sufficient condition of Turing's instability for the system \eqref{eq:model}.
The critical numbers $L_{1}$ and $D_{v}^{*}$ reflecting Turing's instability, which are determined by diffusion coefficient $D_{u}$ and $D_{v}$, are derived. These critical numbers $L_{1}$ and $D_{v}^{*}$ can clearly tell us in which condition the system can generate the Turing's instability, and which type of transition it has. The study in this paper can also point out that the phase transition of Turing's instability has two types-single real eigenvalue transition and double real eigenvalues transition.

%%%%%%%%%%%%%%%%%%%%%%%%%%%%%%%%%%%%%%%%%%%%%%%%%%%%%%%%
%%%%%%%%%%%%%%%%%%%%%%%%%%%%%%%%%%%%%%%%%%%%%%%%%%%%%%%%

\section{Statement of the problem and main results}\label{sec:2}
\subsection{Turing's instability}\label{sec:Turing}
The kind of instability driven by diffusion, which can be widely found in the activator-inhibitor systems, is so-called Turing's instability. Turing firstly noticed that such activator-inhibitor system can generate a stationary pattern, if the inhibitor diffuses faster than the activator. For over a half century, Turing's instability and pattern formation has been studied widely. In this paper, we study the following system
\bea
\begin{cases}
   \pd{u}{t}=D_{u}\Delta u + f(u,v),\\[10pt]
   \pd{v}{t}=D_{v}\Delta  v + g(u,v).
\end{cases}\label{eq:model}
\eea
where $u$ is the concentration of a short-range autocatalytic substance (i.e. activator), and $v$ is the long-range antagonist (i.e. inhibitor) of $u$.

 We obtain the necessary and sufficient condition of Turing's instability for the system \eqref{eq:model}.
The critical numbers $L_{1}$ and $D_{v}^{*}$ reflecting Turing's instability, which are determined by diffusion coefficient $D_{u}$ and $D_{v}$, are derived. These critical numbers $L_{1}$ and $D_{v}^{*}$ can clearly tell us in which condition the system can generate the Turing's instability, and which type of transition it has. The study in this paper can also point out that the phase transition of Turing's instability has two types-single real eigenvalue transition and double real eigenvalues transition.

%%%%%%%%%%%%%%%%%%%%%%%%%%%%%%%%%%%%%%%%%%%%%%%%%%%%%%%%

\subsection{A glance into main results}
In \cite{Maini2000} studying Turning's instabilities, the authors derived the following necessary condition for Turing's instability. In their context, they considered the following n-dimensional reaction-diffusion system:
\bea
\begin{aligned}
&\pd{u_j}{t} = D_j \Delta u_j + f_j(u), \\
&u = (u_1, u_2, ... u_n), x \in \mathbb{R}, t\in [0,T), T>0 \nonumber
\end{aligned}\label{eq:Maini_1}
 \eea
 with initial condition
 \bea
 u_j(x,0) = u_{j0}(x), \quad j = 1,2,...,n.
 \eea
and the associated linearized equation of \eqref{eq:Maini_1} is 
\bea
u_t = Au, \quad u(x,0) = u_0(x)	\label{eq:Maini_2}
\eea
where $A = \{a_{ij}\}_{1\le i,j \le n}$ then we have the following result with respect to \eqref{eq:Maini_2} as follows:
\begin{thm}[Satnoianu et. al 2000] \label{thm:maini}
If the kinetic system \eqref{eq:Maini_2} of the problem \eqref{eq:Maini_1} is s-stable then no
Turing bifurcation is possible from the uniform steady state solution $u_s$ for any $n \ge 1$.
\end{thm}
In particular, when $n=2$, the necessary conditions for Turing's instability are:

\bea
&&a_{11} + a_{22} < 0, \quad a_{11}a_{22} - a_{12}a_{21}> 0, \\
&&D_1 a_{11} + D_2 a_{22} > \sqrt{D_1 D_2 (a_{11}a_{22} - a_{12}a_{21})} > 0; 
\eea
which is the standard result of Turing's instability \cite{Turing1952,MurrayBook2001}. One of our main contributions is to generalize the above result, and to give a \textit{sufficient and necessary} condition for Turing's instability when $n=2$:
\begin{thm}[Theorem \ref{thm:main}]\label{thm:1} 
If
\begin{equation}
det(A)>0,~~Tr(A)<0
\end{equation}
 where $A$ is as in (\ref{ex:A}), and
 \bea
 &&L=-\frac{det(A)}{D_{u}D_{v}}
  +\frac{D^{2}}{4D_{u}^{2}D_{v}^{2}}, \\
 && D=D_{v}a_{11}+D_{u}a_{22}, \\
 &&det(A)=a_{11}a_{22}-a_{12}a_{21}.
 \eea
   then we get two conclusions:\\
(1) If $D \leq 0$ and $L \leq 0 $, system \eqref{eq:model} is Turing stable.\\
(2) If $L > 0 $, system \eqref{eq:model} generates Turing's instability if and only if
there exists an eigenvalue $\lambda_{k}$ of $ -\Delta $ such that
$-\sqrt{L}<\lambda_{k}-\frac{D}{2D_{u}D_{v}}<\sqrt{L}$.
 \end{thm}
 
Besides, the critical parameter $D_{v}^{*}$ defined in \eqref{eq:ex_critical} can be found as follows
 \begin{thm} [Theorem \ref{thm:main_cont}]
 Let $D_{u}$ be fixed, and (\ref{neq:turing_transition}) holds true,
then the critical value of the transition for  system \eqref{eq:model} is
$D_{v}^{*}$, i.e, the system \eqref{eq:model} generates the Turing's instability if and only if $D_{v} > D_{v}^{*}$.
\end{thm}
 
Using a new technique called dynamic bifurcation theory as in Section \ref{sec:gemetric}., we will also provide a geometric explanation for Turing's instability.

Specifically, we give an application of the above result to study solution behaviors of Schnakenberg reaction diffusion system \cite{Schnakenberg1979}. This simple system can produce a great amount of different dynamic patterns with modifications of several simple parameters. Nevertheless, we propose the following necessary and sufficient condition for Turing's instability using \textit{center manifold reduction}.

 \begin{thm}[Theorem \ref{thm:sch_1}]
Assume $\frac{b-a}{a+b}-(a+b)^{2}<0$\\
(1) if
$-\frac{\sqrt{4d}}{r(a+b)}<d\frac{b-a}{a+b}-(a+b)^{2}<0$, i.e.
\bea
\max{(d,1)}(b-a) < (a+b)^3 < (d+\frac{\sqrt{4d}}{r})(b-a)
\eea
hen system (\ref{eq:sch}) is Turing stable.\\
(2) if
$d\frac{b-a}{a+b}-(a+b)^{2}>\frac{\sqrt{4d}}{r(a+b)}$ or $d\frac{b-a}{a+b}-(a+b)^{2}<-\frac{\sqrt{4d}}{r(a+b)}$
\bea
b-a < (a+b)^3 < (d-\frac{\sqrt{4d}}{r})(b-a)
\eea
\bea
\text{or~}\max{(1,d+\frac{\sqrt{4d}}{r})}(b-a) < (a+b)^3 
\eea
then system (\ref{eq:sch}) generates Tuing's instability if and only if there exists an eigenvalue $\lambda_{k_{0}}$ of $-\Delta$ such that $\beta_{-}<\lambda_{k_{0}}<\beta_{+}$.
\end{thm}

Further more, we discussed the case of single and double eigenvalue bifurcation respectively. For example, in the case of double eigenvalue bifurcation, we are able to get the explicit expression for the system above
\begin{thm}[Theorem \ref{thm:sch_2}]
The system \eqref{eq:sch}  has a stable steady state $(0,0)$ for $d<d_{0}$, and bifurcates a new steady state $(x_{0},y_{0})$  for $d>d_{0}$, if and only if the following conditions hold true
\bea
  A_{1}+A_{4}<0,
  A_{1}A_{4}-A_{2}A_{3}>0,
\eea
where
\beaa
  &A_{1}=\beta_{\lambda_{i_{1}}1}+2a_{20}x_{0}+a_{11}y_{0}, &A_{2}=2y_{0}a_{02}+a_{11}y_{0},\\
  &A_{3}=2b_{02}x_{0}^{2}+b_{11}y_{0},&A_{4}=\beta_{\lambda_{i_{2}}1}+2b_{20}y_{0}+b_{11}x_{0}.
\eeaa
and there is a stable attractor $u^{d}$ of the system bifurcated from $d>d_{0}$  as follows
\beaa
  u^{d}=x_{0}\xi e_{\lambda_{i_{1}}}+y_{0}\eta e_{\lambda_{i_{2}}}
 \eeaa
\end{thm}
besides, we are also able to provide a geometric interpretation of the above bifurcation results, using phase plane diagram.

%%%%%%%%%%%%%%%%%%%%%%%%%%%%%%%%%%%%%%%%%%%%%%%%%%%%%%%%
%%%%%%%%%%%%%%%%%%%%%%%%%%%%%%%%%%%%%%%%%%%%%%%%%%%%%%%%
\section{Key results in dynamic transition theory}\label{sec:3}
%%%%%%%%%%%%%%%%%%%%%%%%%%%%%%%%%%%%%%%%%%%%%%%%%%%%%%%%

\subsection{Dynamic transition theory -- basic setup}
Transitions are found throughout our everyday lives. Before studying details of transition problems, we need a good understanding about the nature world in advance. The laws of nature are
usually represented by differential equations, which can be regarded as dynamical
systems -- both finite and infinite-dimensional. In this section, we briefly introduce the key ingredients of dynamic transition theory developed by Ma and Wang \cite{Ma2015}.

We start with the reaction-diffusion system proposed in section \ref{sec:Turing}, with Dirichlet/Neumann boundary condition. Without loss of generality, suppose the steady state $(u_{0},v_{0})$ =(0,0), and take Taylor expansion of \eqref{eq:model} at $(0,0)$ as follows:
\bea
\begin{cases}
   \pd{u}{t}=D_{u}\Delta u+a_{11} u+a_{12}v+g_{1}(u,v),\\[10pt]
   \pd{v}{t}=D_{v}\Delta v+a_{21} u+a_{22}v+g_{2}(u,v),
\end{cases}\label{eq:linearized}
\eea
where
\begin{equation}\label{ex:A}
A:=
\left(
\begin{array}{cc}
  a_{11}& a_{12}\\
  a_{21}& a_{22}
\end{array}
\right)
=\left(\begin{array}{cc}
   f_{u}(0,0) & f_{v}(0,0) \\
   g_{u}(0,0) & g_{v}(0,0)
 \end{array}\right).
\end{equation}
Let
\begin{align}\label{202}
 &H =L ^{2}(\Omega, \mathbb{R}^{2}),\\
\label{203}
   &H_{1}=\{(u,v)\in H ^{2}(\Omega,\mathbb{R}^{2}):\frac{\partial u}{\partial n}=0 ~\text{on}~  \partial \Omega \}.
\end{align}
Define operator $ L_{\mu}=A_{\mu}+B_{\mu}$ and $ G_{\mu}$ as  follows 
\begin{align}\label{205}
&L_{\mu},G_{\mu} : H_{1}\rightarrow H,  \\
 \label{206}&
  A_{\mu}u=(D_{u}\Delta u,D_{v}\Delta v),
\end{align}

\begin{equation}\label{207}
B_{\mu}u=
\left(
\begin{array}{cc}
  a_{11}& a_{12}\\
  a_{21}& a_{22}
\end{array}
\right)u,
\end{equation}
\begin{equation}\label{208}
G_{\mu}(u)=(g_{1}(u),g_{2}(u)).
\end{equation}
So $L_{\mu}$ and $G_{\mu}$ are respectively the linear and nonlinear part of \eqref{eq:linearized}. Therefore, system \eqref{eq:linearized} can be rewritten as
\begin{eqnarray}\label{eq:operator}
\left\{
   \begin{array}{ll}
   \pd{u}{t}=L_{\mu}u + G_{\mu}(u),
   \\u(0)=\varphi,
\end{array}
\right.
\end{eqnarray}
where $u=(u,v)$ , $\mu=(D_{u},D_{v},r)\in \mathbb{R}_{3}^{+}=\{ (x_{1},x_{2},x_{3})\in \mathbb{R}^{3}:x_{i}\geq 0,i=1,2,3 \}$
and $ G_{\mu}( u )=o(\|u\|_{X_{\alpha}})$.\\

In the next section, we will introduce our main result. We start with solving eigenvalue problem of the system \eqref{eq:operator}.

\subsection{Exchange of stability} \label{sec:exchange}
Bifurcation theory in ODE is already well-developed. It is well-know in the context of ODE bifurcation theory that a bifurcation happens when the max real part of the eigenvalues of the right hand side Jacobian matrix changes sign when some parameter $\mu$ passes a critical value $\mu_0$.

A natural question arises: how can we extend the ODE bifurcation theory to PDE case? Without too many doubts, the answer is still related to max real part of the eigenvalues. However, we need to handle with the potential difficulty that PDE problems are of infinitely dimensional.

Luckily enough, as we will see in Section \ref{sec:eigenvalue}, under certain conditions, a large group of infinite dimensional operators have nice spectral properties. Therefore, for this kind of systems, we can introduce the \textit{Exchange of stability} property (first coined by Davis \cite{Davis1969} and formally explored by Ma and Wang \cite{Ma2005_evolution,Ma2015}),.
\begin{definition}[Principle of Exchange of Stability, PES]\label{def:PES}
Consider the system \eqref{eq:operator} let $\{\beta_i \in \mathbb{C} | i = 1,2,....\}$ be all eigenvalues of $L_\mu$ (counting multiplicities), and suppose that they satisfy
\bea
&Re~ \beta_i(\mu)
\begin{cases}
	< 0 \quad \textbf{if~} \mu < \mu_0,\\
	= 0 \quad \textbf{if~} \mu = \mu_0,\\
	> 0 \quad \textbf{if~} \mu > \mu_0,
\end{cases} & 1\le i \le m \\[10pt]
&Re~ \beta_i(\mu) \neq 0 \hspace{80pt} & j\ge m+1
\label{eq:exchange}
\eea
for some $\mu_0\in\mathbb{R}^1$. Then we say the system \eqref{eq:operator} satisfies Principle of Exchange of Stability or PES for short.
\end{definition}
Definition. \ref{def:PES} provided us a natural way to divide the eigenvalues of \eqref{eq:operator} into two different ways. More specifically, let $\{e_1(\mu),e_2(\mu),...,e_m(\mu) \}$ and $\{e^*_1(\mu),e^*_2(\mu),...,e^*_m(\mu) \}$ be the eigenvectors of $L_\mu$ and its conjugate operator $L^*_\mu$ corresponding to the eigenvalues $\beta_i(\mu), 1\le i \le m$ respectively. According to Theorem 3.4 in \cite{Ma2005} (to my knowledge, this had been the first time that the theorem ever appeared)
\begin{thm}[Spectral decomposition of a linear completely continuous field]
	\label{thm:spectral}
Let $L = A + B : X_1 \to X$ be a linear completely continuous field, then we have following results:
\begin{enumerate}
	\item if $\{\beta_k | k\ge 1\} \in\mathbb{C}$ are eigenvalues of $L$ and $\{e_1,e_2,...,e_m \}$,  $\{e^*_1,e^*_2,...,e^*_m \}$ be the corresponding eigenvectors of $L_\mu$ and its conjugate operator $L^*_\mu$, then 
	\bea
	\braket{e_i,e^*_j} = \delta_{ij}
	\eea
	
	\item $X$ can be decomposed into the following direct sum
	\bea
	\begin{cases}
		X = \bar{E_1} \oplus E_2,\\
		E_1 = \text{span~}\{e_k | k\ge 1\} \cap X,\\
		E_2 = \{\bv \in X | \braket{\bv, e_k^*}_X = 0  \forall k\ge 1\};
	\end{cases}
	\eea
	
	\item $E_1$ and $E_2$ are invariant spaces of $L^{-1}$ and
	\bea
	\lim_{n\to\infty} ||L^{-1}\bv||_X^{1/n} = 0, \forall \bv \in E_2;
	\eea
	
	\item Let $\{\gamma_1,\gamma_2,...,\gamma_k\} \subset\mathbb{C}$ be eigenvalues of $L^{-1}$ (counting multiplicity) in the order $|\gamma_1| \ge |\gamma_2| \ge ...\ge |\gamma_k|$, $\{f^*_1,f^*_2,...,f^*_k \}\subset{X\otimes\mathbb{C}} \ $ be corresponding eigenvalues ($X\otimes\mathbb{C}$ is the complexification of $X$) of $(L^{-1})^* = L^{*-1}$, and let $E_k^* = \text{span}\{f^*_1,f^*_2,...,f^*_k \}$ ($E_k^* = \emptyset$ if $k=0$). If
	\bea
	\rho_{k+1} := \sup_{u\in X\otimes\mathbb{C}} \lim_{n\to\infty}|\braket{L^{-n}u,u}^{1/n} > 0|
	\eea
	then there is an eigenvalue $gamma_{k+1}\in \mathbb{C}$ of $L^{-1}$ with $|\gamma_{k+1}| = \rho_{k+1}$ and $|\gamma_{k+1}|\le|\gamma_k|$.
\end{enumerate}
\end{thm}
\begin{spf}
Conclusion 1 and 2 are analogs to Jordan's Decomposition of a finite dimension matrix (Theorem 3.3 of \cite{Ma2005}); Conclusion 3 is a direct outcome of orthogonality between the two spaces $E_1$ and $E_2$; Finally, combining Conclusion 1-3 together, we can decompose any vector $\bullet\in X\otimes\mathbb{C}$ in a proper way, yielding Conclusion 4. Details of the proof can be found in \cite{Ma2005}.
\end{spf}
Applying Theorem \eqref{thm:spectral} to \eqref{eq:operator}, we see that
\bea
X_1 = E_1 \oplus E_2, \quad X = \bar{E_1} \oplus E_2,
\eea
where
\bea
\begin{cases}
	E_1 = \text{span~}\{e_k | k\ge 1\} \cap X,\\
	\bar{E_1} = \text{~closure~of~}E_1,\\
	E_2 = \{\bv \in X | \braket{\bv, e_k^*}_X  \forall k\ge 1\}.
\end{cases}
\eea

Another important element in dynamic bifurcation theory is how to classify transition types. This can be done by making use of Principle \ref{def:PES} (PES). The following theorem is a basic principle of transitions from equilibrium states. It provides sufficient conditions and a basic classification for transitions of nonlinear dissipative systems (see Theorem 2.1.3 of \cite{Ma2015}).
\begin{thm}[Classification of transition types]
Consider the system \eqref{eq:operator} with , if it satisfies PES then it always undergoes a dynamic transition from $(\bu_0, \mu) = (0,\mu_0)$ (w.l.o.g, we can set $0$ to be its steady state), and there is a neighborhood $U\subset X$ of $\bu = 0$ such that the transition is one of the following three types:
\begin{enumerate}
\item \textbf{Continuous Transition}: there exists an open and dense set $\tilde{U}_{\mu}\subset U$ such that for every $\phi \in \tilde{U}_{\mu}$ , the solution $\bu_{\mu}$ of \eqref{eq:operator} satisfies
\beaa
\lim_{\mu\to\mu_0} \limsup_{t\to\infty}||\bu_{\mu}(t,\phi)|| = 0.
\eeaa
\item \textbf{Jump Transition}: for every $\mu_0 < \mu < \mu_0  \epsilon$ with some $\epsilon > 0$, there is an open and dense set $\tilde{U}_{\mu}\subset U$ and a number $\delta > 0$ independent of $\mu$ such that for any $\phi \in \tilde{U}_{\mu}$, 
\beaa
\limsup_{t\to\infty}||\bu_{\mu}(t,\phi)|| \ge \delta > 0,
\eeaa
\item \textbf{Mixed Transition}:  for every $\mu_0 < \mu < \mu_0  \epsilon$ with some $\epsilon > 0$, $U$ can be decomposed into two open (not necessarily connected) sets $U_1^{\mu}$ and $U_2^{\mu}$: $\bar{U} = \bar{U}^{\mu}_1	\cup \bar{U}^{\mu}_2$, $\o = U^{\mu}_1 \cap U^{\mu}_2$ such that
\beaa
&\lim_{\mu\to\mu_0} \limsup_{t\to\infty}||\bu_{\mu}(t,\phi)|| = 0, & \forall \phi \in U_1^{\mu}\\
&\limsup_{t\to\infty}||\bu_{\mu}(t,\phi)|| \ge \delta > 0 & \forall \phi \in U_1^{\mu},
\eeaa
here $U_1^{\mu}$ and $U_2^{\mu}$ are called metastable domains.
\end{enumerate}
\end{thm}
 \begin{figure}[!ht]
	\label{fig:GS_2d}
 	\centering
 	\includegraphics[width = .9\textwidth]{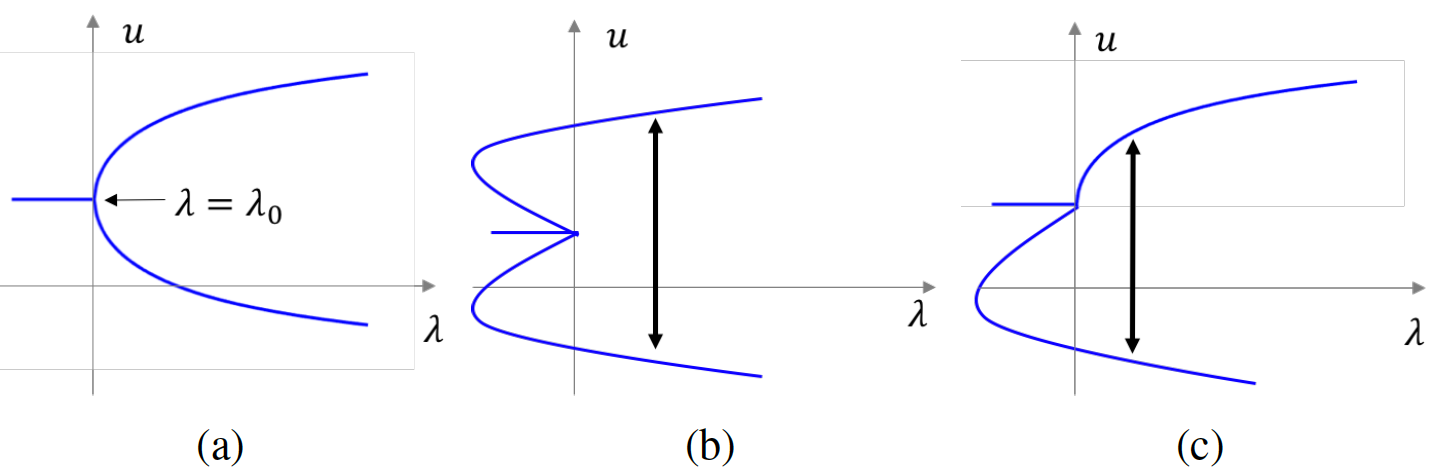}
 	\caption{\textbf{An illustration of different transition types}. (a): continuous; (b): jump and (c): mixed.}
 \end{figure}

%%%%%%%%%%%%%%%%%%%%%%%%%%%%%%%%%%%%%%%%%%%%%%%%%%%%%%%%
\subsection{The eigenvalue problem}\label{sec:eigenvalue}
Since the operator $L_{\mu}$ is defined on a Hilbert space, and its range is on another Hilbert space, we need to introduce how to define and solve eigenvalues problem of a infinite dimensional space.

Denote $\mu_{k}$ and $e_{k}$ to be eigenvalues and eigenvectors of $-\Delta$ on $H$ respectively, due to the spectral decomposition theory \cite{Ma2005}, $\{e_{k}\}$ is a base of $H_{1}$, where $e_{k}$ and $\mu_{k}$ satisfy the following equations
\bea
\begin{cases}
   -\Delta e_{k}=\mu_{k}e_{k},
   \\[5pt] e_{k}=0\text{~or~}\pd{e_k}{t}=0~on~\partial\Omega, \quad k \ge 0
\end{cases}
\eea
and $0=\mu_{0}<\mu_{1}<\mu_{2}< \ldots$ . By the eigenvalue $ \{\mu_{k}\}$ of $-\Delta$,
 we can get the eigenvalues $\beta_{k}^{(i)}$ and the eigenvectors of $L_{\mu}$ satisfying the following equations
\begin{eqnarray*}
\left\{
   \begin{array}{ll}
   D_{u}\Delta( \varepsilon_{1i}^{k}e_{k})+a_{11} \varepsilon_{1i}^{k}e_{k}+a_{12}\varepsilon_{2i}^{k}e_{k}=\beta_{k}^{(i)}\varepsilon_{1i}^{k}e_{k}
   \\[5pt] D_{v}\Delta( \varepsilon_{2}^{k}e_{k})+a_{21} \varepsilon_{1i}^{k}e_{k}+a_{22}\varepsilon_{2i}^{k}e_{k}=\beta_{k}^{(i)}\varepsilon_{2i}^{k}e_{k}
\end{array}
\right.
\end{eqnarray*}
where $k=0,1,2,3,4,\ldots$, and $i=1,2$.

Denote
\begin{equation}\label{neq:turing_transition}
M_{\mu_{k}}=
\left(
  \begin{array}{cc}
    a_{11}-\mu_{k}D_{u} & a_{12} \\
    a_{21} & a_{22}-\mu_{k}D_{v} \\
  \end{array}
\right),
\end{equation}
then all the eigenvalues $\beta_{k}^{(i)}$ of $L_{\mu}$ are all eigenvalues of $ M_{\mu_{k}} $. It is easy to see that $\beta_{k}^{(i)}(i=1,2)$ are the two solutions of the following equation
\begin{align}\label{root}
  \mu^{2}-Tr(M_{\mu_{k}})\mu+
  det(M_{\mu_{k}})=0,
 \end{align}
and
\begin{align*}
&\beta_{k}^{(2)}=\frac{Tr(M_{\mu_{k}})+
\sqrt{Tr(M_{\mu_{k}})^{2}-4det(M_{\mu_{k}})}}{2},\\
&\beta_{k}^{(1)}=\frac{Tr(M_{\mu_{k}})-
\sqrt{Tr(M_{\mu_{k}})^{2}-4det(M_{\mu_{k}})}}{2},
\end{align*}
where
\begin{align}\label{eq:turing_trace}
   &Tr(M_{\mu_{k}})=a_{11}+a_{22} -\mu_{k}(D_{u}+D_{v}),
  \\ &\label{eq:turing_det}det(M_{\mu_{k}})=D_{u}D_{v}\mu_{k}^{2}-(a_{22}D_{u}
  +a_{11}D_{v})\mu_{k}+det(A).
\end{align}

%%%%%%%%%%%%%%%%%%%%%%%%%%%%%%%%%%%%%%%%%%%%%%%%%%%%%%%%

\subsection{Center manifold reduction}\label{sec:cmr}
In physical science, it is often crucial to determine the asymptotic behavior of a system at the critical threshold. For this purpose and for determining the structure of the local attractor representing the transition states, the most natural approach is to project the underlying system to the space generated by the most unstable modes, preserving the dynamic transition properties. This is achieved with center manifold reduction \cite{TemamBook1997,WigginsBook2003,GuckenheimerBook2013,Ma2015}.

Let $X_1$ and $X$ be two Banach spaces with $X_1 \subset X$ a dense and compact inclusion. Consider the following one-parameter nonlinear evolution equation again (in the remaining of this chapter, we do not distinguish $G(\cdot, \mu)$ and $G_{\mu}(\cdot)$):
\bea
\begin{aligned}
&\pd{\bu}{t} = L_{\mu} \bu +G(\bullet,\mu),\\
&\bu(0,t) = \bu_0(t), \quad \bx \in \Omega, \mu \in \mathbb{R}^1
\end{aligned}\label{eq:cmr}
\eea
where $L_{\mu} = A_\mu + B_\mu : X_1 \to X$ is a \textit{linear completely continuous fields}, (i.e. $A$ is a linear homeomorphism and $B_\mu$ a linearly compact operator) depending continuously on $\mu$. Suppose that the system \eqref{eq:cmr} satisfies PES \ref{def:PES},  by the \textit{Spectral Decomposition Theorem} \ref{thm:spectral},  $L_{\mu}$ can be decomposed into $L_{\mu} = L^{(1)}_{\mu} + L^{(2)}_{\mu}$ such that for every $\mu$ sufficiently close to $\mu_0$,
\bea
\begin{aligned}\label{eq:cmr_cond}
L^{(1)}_{\mu} = L_{\mu}|_{E_1} : E_1 \to \tilde{E}_1,\\
L^{(2)}_{\mu} = L_{\mu}|_{E_2} : E_2 \to \tilde{E}_2,
\end{aligned}
\eea
where $\tilde{E}_i$ is the completion of $E_i \subset X_1$ in $X$, the eigenvalues of $L^{(1)}_{\mu}$ have nonnegative real parts and $L^{(2)}_{\mu}$ have negative real parts at $\mu = \mu_0$. Therefore, \eqref{eq:cmr} can be written as
\bea
\begin{aligned}\label{eq:cmr_decompose}
\pd{\bu_1}{t} =  L^{(1)}_{\mu} \bu_1 + G_1(\bu_1, \bu_2, \mu),\\
\pd{\bu_2}{t} =  L^{(2)}_{\mu} \bu_2 + G_2(\bu_1, \bu_2, \mu).
\end{aligned}
\eea
where $\bu = \bu_1 + \bu_2$, $\bu_1 \in E_1$, $\bu_2 \in E_2$, $G_i(\bu_1, \bu_2, \mu) = P_i G(\bu, \mu)$ and $P_i : X \to E_i$ are canonical projections.

Clearly, in finite dimensional case, $L_{\mu}$ is just a $n \times n$ matrix, denote it by $A$. So \eqref{eq:cmr_decompose} can be expressed as
\bea
\begin{aligned}\label{eq:cmr_decompose_finite}
\frac{d \bu_1}{dt} =  A^{(1)}_{\mu} \bu_1 + G_1(\bu_1, \bu_2, \mu),\\
\frac{d \bu_2}{dt} =  A^{(2)}_{\mu} \bu_2 + G_2(\bu_1, \bu_2, \mu).
\end{aligned}
\eea

The followings are the well-known \textit{center manifold theorems}. Theorem \ref{thm:cmr_finite} can be found in standard books for bifurcation theory, e.g. \cite{WigginsBook2003,GuckenheimerBook2013}. Theorem \ref{thm:cmr_infinite} can be found in \cite{TemamBook1997,Ma2005}.

\begin{thm}[Finite dimensional case]\label{thm:cmr_finite}
Suppose that all the eigenvalues of A have non-negative real parts, and all the eigenvalues of B have negative (or positive) real parts. Then, for the system \label{eq:cmr_decompose_finite} with the condition that all eigenvalues of $A^{(1)}$ is non-negative (resp. non-positive) and eigenvalues of $A^{(2)}$ negative (resp. positive), then there exists a $C^r$ function
\beaa
\Phi(\cdot, \mu) : \Omega \to \mathbb{R}^{n-m}; \Omega \subset \mathbb{R}^{m} \text{~a~neighborhood~of~} \bx = 0,
\eeaa
such that $\Phi(\bx, \mu)$ is continuous on $\mu$ and
\begin{enumerate}
\item $\Phi(0, \mu) = 0, \pd{\Phi}{x}(0,\mu) = 0$;
\item the set
\beaa
M_{\mu} := \{(\bx, \by) | \bx\in \Omega\subset \mathbb{R}^{m}, \by = \Phi(\bx, \mu)\}
\eeaa
called the \text{(local) center manifold}, is a local invariant manifold of \eqref{eq:cmr_decompose_finite};
\item if $M_{\mu}$ is positive invariant (or negative invariant), namely $\bz(t, \phi) \in M_{\mu}(\bz(-t,  \phi) \in M_{\mu}), \forall t \ge 0 $provided $\phi \in M_{\mu}$, then $M_{\mu}$ is an attracting set of \eqref{eq:cmr_decompose_finite}(or a repelling set) , i.e. there is a neighborhood $U \subset \mathbb{R}^n$ of $M_{\mu}$, as $\phi \in U$, we have 
\beaa
\lim_{t\to\infty (\text{~or~}\infty)}dist(\bz(t,\phi), M_{\mu}) = 0
\eeaa
where $\bz(t,\phi) = \{\bx(t,\phi),\by(t,\phi)\}$ is the solution of (5.2.3) with the initial condition $z(0, \phi) = \phi$.
\end{enumerate}
\end{thm}

%For infinite dimensional case, we need $X_1$ and $X$ to be Hilbert spaces. We further assume that the spaces $X_1$ and $X$ can be decomposed into
%\bea
%\begin{cases}
%X_1 = E^{\mu}_1 \oplus E^{\mu}_2,	\quad dim  E^{\mu}_1 < \infty \\
%X_1 = \tilde{E}^{\mu}_1 \oplus \tilde{E}^{\mu}_2,	\quad  E^{\mu}_1 =\tilde{E}^{\mu}_1\\
%E^{\mu}_2 = \text{~closure~of~} E^{\mu}_2 in X,
%\end{cases}
%\eea
%and for $\mu$ near $\mu_0$, $L_\mu$ can be decomposed into $L = \mathcal{L}_1^\mu \oplus \mathcal{L}_2^\mu$, where the eigenvalues of $\mathcal{L}_1^\mu, \mathcal{L}_2^\mu$ have non-negative and negative real part respectively, and such that
%\bea
%\begin{cases}
%\mathcal{L}_1^\mu = L_\mu|_{E_1^{\mu}} : E_1^{\mu} \to \tilde{E}_1^{\mu}\\
%\mathcal{L}_2^\mu = L_\mu|_{E_2^{\mu}} : E_2^{\mu} \to \tilde{E}_2^{\mu}
%\end{cases}
%\eea

\begin{thm}[Infinite dimensional case]\label{thm:cmr_infinite}
Suppose \eqref{eq:cmr}-\eqref{eq:cmr_decompose}, and assume $X_1$ and $X$ are Hilbert spaces, then there exists a neighborhood of $\mu$ given by $|\mu-\mu_0| < \delta$ for some $\delta > 0$, a neighborhood $O_{\mu} \subset E^{\mu}_1$ of $x = 0$, and a $C^1$ function $\Phi(\cdot,X) : O_{\mu} \to E^{\mu}_2(\alpha)$ depending continuously on $\mu$, where $\bar{E}^{\mu}_2(\alpha)$ is the completion of $E^{\mu}_2(\alpha)$ in the $X_{\alpha}$-norm, with $0 < \alpha < 1$ such that
\begin{enumerate}
\item $\Phi(0, \mu) = 0, \pd{\Phi}{x}(0,\mu) = 0$;
\item the set
\beaa
M_{\mu} := \{(\bx, \by) | \bx\in O_{\mu}, \by = \Phi(\bx, \mu) \in E^{\mu}_2(\alpha)\}
\eeaa
called the \text{(local) center manifold}, is a local invariant manifold of \eqref{eq:cmr};
\item $(\bx_{\mu}(t), \by_{\mu}(t))$ is a solution to \eqref{eq:cmr}, then there is a $\beta_\mu > 0 $ and $C_\mu > 0$ with $k_\mu$ depending on $(\bx_{\mu}(0), \by_{\mu}(0))$ such that 
\bea
||\by_{\mu}(t) - \Phi(\bx_{\mu}(t),\mu)||_{X} \le C_{\mu} \exp{(-\beta_\mu t)}
\eea
\end{enumerate}
\end{thm}

\begin{rmk}[1]
It is noticeable that, Theorem \ref{thm:cmr_infinite} only works when $X$ and $X_1$ are Hilbert spaces, due to their spectral properties. By the \textit{spectral decomposition theorem} \ref{thm:spectral}, however, we can extend the result of Theorem \ref{thm:cmr_infinite} to any totally continuous field $L$ in Banach spaces. Details can be found in Ma and Wang's book \cite{Ma2005}.
\end{rmk}

\begin{rmk}[2]
Both Theorem \ref{thm:cmr_finite} and \ref{thm:cmr_infinite} ascertained the existence of center manifold functions for finite and infinite dimensional dynamical systems. However, they do not provide a explicit way for constructing the central manifold. A systematic construction can be found in Section 3 of \cite{Ma2005} and Appendix A of \cite{Ma2015}, which is skipped in this context.
\end{rmk}

While Theorems \ref{thm:cmr_finite} and \ref{thm:cmr_infinite} show the significance of a center manifold function, they do not tell how to explicitly find these functions. In general, due to our knowledge, the only systematic way of calculating center manifold functions are to assume polynomial structures of $Gs'$ in \eqref {eq:cmr}, or to use Taylor's expansion near some steady state and critical parameter. For detailed calculation, please refer to Chapter 18 of \cite{WigginsBook2003} for finite dimensional cases, and Section 3.2 of \cite{Ma2005} or Appendix A of \cite{Ma2015}. For example, if $G$ in \eqref{eq:cmr} has the following form
\bea\label{eq:ex_Gn}
G(\bu,\mu) = \sum_{n=k}^\infty G_n(\bu,\mu),
\eea
for some $k\ge 2$, where $G_n : \prod_{i=1}^n X_1 \to X$ is an $n$-multiple linear mapping, and $G_n(\bu,\mu):=G_n(\bu,...,\bu,\mu)$. Then we have

\begin{thm}[Theorem 3.8 of \cite{Ma2005}]
If $L_\mu$ is a sectorial operator and $G$ satisfies \eqref{eq:ex_Gn}, then the center manifold function $\Phi(\bx,\mu)$ in Theorem \ref{thm:cmr_infinite} can be expressed as
\bea\label{eq:cmr_fun}
\Phi(\bx,\mu) = (-L^{(2)}_{\mu})^{-1} P_2 G_k(\bx,\mu) + O(|Re~\beta(\mu)|\cdot ||\bx||^k) + o(||\bx||^k)
\eea
where $ P_2 : X_1 \to \tilde{E}_2$ is the canonical projection, $\bx \in E_1$ and $\beta(\mu) = (\beta_1(\mu),...,\beta_m(\mu))$ are the eigenvalues of $L^{(1)}_{\mu}$.
\end{thm}

By \eqref{eq:cmr_fun} or equation (A.1.14) of \cite{Ma2015}, it can be calculated that
\bea
\begin{aligned}\label{eq:cmr_higher}
&\Phi(\bx,\mu) = \sum_{j=m+1}^{\infty} \Phi_j(\bx,\mu) e_j\\
&\Phi_j(\bx,\mu) = -\frac{1}{\beta_j(\mu)<e_j,e_j^*>_{X,X^*}} <G_j(\bx,\mu),e^*_j>_{X,X^*} + o(k)
\end{aligned}
\eea
where $\bx = \sum_{i=1}^m x_i e_i$ and $\beta(\mu) = (\beta_1(\mu),...,\beta_m(\mu))$, $<\cdot,\cdot>_{X,X^*}$ denotes the dual product between the function space $X$ and its dual $X^*$.

%Let $\mathcal{X}_1$ and $\mathcal{X}$ be two Banach spaces, and $\mathcal{X}_1 \in \mathcal{X}$ a dense and compact inclusion. Consider the following nonlinear evolution equation:
%\bea
%&\frac{d\bu}{dt} = L_{\mu} \bu + G(\bu, \mu), \qquad \mu \in \mathbb{R}^1,\\
%&u(0) = u_0.
%\eea
%where 
%\bea
%\begin{aligned}\label{eq:cmr_cond}
%&&L_{\mu} = -A + B_\mu : \mathcal{X}_1 \rightarrow \mathcal{X} \\
%&A: \mathcal{X}_1 \rightarrow \mathcal{X} & \text{a~linear~homeomorphism,}\\
%&B_{\mu}: \mathcal{X}_1 \rightarrow \mathcal{X} & \text{parameterized~linear~compact~operators.}
%\end{aligned}
%\eea
%
%Consider the case in which $L_{\mu} : $

%%%%%%%%%%%%%%%%%%%%%%%%%%%%%%%%%%%%%%%%%%%%%%%%%%%%%%%%
%%%%%%%%%%%%%%%%%%%%%%%%%%%%%%%%%%%%%%%%%%%%%%%%%%%%%%%%

\section{Stability and dynamic transition of Turing's systems}\label{sec:turing}
In this section, we study the stability and dynamic transition behavior of the  Turing's system \eqref{eq:model}.

\subsection{Critical parameters of the Turing's system}
The following theorem is the necessary and sufficient condition for Turing's instability of system \eqref{eq:model}. Let
\begin{align}\label{eq:ex_transition_L}
  L=-\frac{det(A)}{D_{u}D_{v}}
  +\frac{D^{2}}{4D_{u}^{2}D_{v}^{2}},
\end{align}
and assume that
\begin{equation}\label{neq:tr_det}
det(A)>0,~~Tr(A)<0
\end{equation}
 where $A$ is as in (\ref{ex:A}), and
 \begin{align}\label{eq:ex_turing_D}
 & D=D_{v}a_{11}+D_{u}a_{22}, \\
  \label{eq:A_det}&det(A)=a_{11}a_{22}-a_{12}a_{21}.
 \end{align}
\begin{thm}\label{thm:main} Under the assumption (\ref{neq:tr_det}), we obtain the following two conclusions:
(1) If $D \leq 0$ and $L \leq 0 $, system (\ref{eq:model} ) is Turing stable.
(2) If $L > 0 $, system \eqref{eq:model} generates Turing's instability if and only if
there exists an eigenvalue $\lambda_{k}$ of $ -\Delta $ such that
$-\sqrt{L}<\lambda_{k}-\frac{D}{2D_{u}D_{v}}<\sqrt{L}$.
\end{thm}
 \begin{proof}
 First, we prove the first assertion.
 Let
\bea\label{eq:turing:center}
\begin{aligned}
h(\lambda,L )&=D_{u}D_{v}\lambda^{2}-(D_{u}a_{22}+D_{v}a_{11})\lambda
+a_{11}a_{22}-a_{12}a_{21} \nonumber \\
&=D_{u}D_{v}[(\lambda-\frac{D}{2D_{u}D_{v}})^{2}-L ]
\end{aligned}
\eea
If $D \leq 0$ and $L \leq 0 $, combining with (\ref{eq:turing_trace} ) and \eqref{eq:turing:center}, it is easy to check that $det(M_{\lambda_{k}})=h(\lambda_{k},L)>0$ and $tr(M_{\lambda_{k}})>0$ holds true for all the eigenvalues $\{\lambda_{k}\}$ of the operator $-\Delta$, that is, all eigenvalues of $L_{\lambda}$ have negative real part, which means that system \eqref{eq:model}is Turing stable.

Secondly, we prove the second assertion.

\textbf{Sufficiency}. The condition $Tr(A)=a_{11}+a_{22}<0 $ and $det(A)=a_{11}a_{22}-a_{12}a_{21}>0$ mean that $(0,0)$ is a stable steady state of the system \eqref{eq:model} without diffusion.
Let the solutions of $h(\lambda,L)=0$ be $k_{\pm}=\frac{D}{2D_{u}D_{v}}\pm\sqrt{L}$. We can deduce from $k_{-}< \lambda_{k} <k_{+}$ that $det(M_{\lambda_{k}})<0$ , that is, there exists eigenvalue $\beta_{k}^{i_{0}}$ of $L_{\lambda}$ such that $ \beta_{k}^{i_{0}}>0$. Then $(0,0)$ is unstable steady state of system \eqref{eq:model}, that is, system \eqref{eq:model}generates Turing's instability.

\textbf{Necessity}. Based on the definition of Turing's instability, $(0,0)$ should be the stable steady state of \eqref{eq:model} without diffusion, that is, $a_{11}+a_{22}<0$ and $a_{11}a_{22}-a_{12}a_{21}>0$ hold true. $(0,0)$ is not the stable steady state of \eqref{eq:model}, which means that there exists $ \beta_{k}^{i_{0}}>0$ of $L_{\lambda}$, such that $h(\lambda_{k},L)= det(M_{\lambda_{k}})<0$. In another word, there is a eigenvalue $\lambda_{k}$ of the operator $ -\Delta$ satisfying $k_{+}<\lambda_{k}<k_{-}$, where $k_{\pm}=\frac{D}{2D_{u}D_{v}}\pm\sqrt{L}$. The proof is complete.
\end{proof}

\begin{remark} \label{rmk:polynomial}
It is also important to see that, when $h(\lambda,L) < 0$, then the system \eqref{eq:model} is Turing unstable. This provide us a way to determine the transition parameter of  \eqref{eq:model} as will be shown in Remark \ref{rmk:calculate}.
\end{remark}

\begin{remark} $D > 0$ and $L > 0 $ means $\frac{D_{v}}{D_{u}}>1 $, which is a necessary condition of the Turing's instability. It means that the inhibitor diffuses faster than the activator  if Turing's instability is achieved.
\end{remark}

The eigenvalues of the operator $-\Delta$ with Dirichlet/Neumann boundary condition in the case that $\Omega=(0,s)$ is as follows
\begin{eqnarray*}
\{\frac{\pi^{2} m^{2}}{s^{2}}: m \in N_{+} \},
\end{eqnarray*}
then we can get the following corollary.
\begin{corollary}
Under the assumptions in Theorem \ref{thm:main}, assume that $D > 0$ and $L > 0$, and taking $\Omega=(0,s)$, then system \eqref{eq:model} generates Turing's instability if and only if there exists a integer $m\in N_{+}$ such that
$-\sqrt{L}<\frac{m^{2}\pi^{2}}{s^{2}}-\frac{D}{2D_{u}D_{v}}<\sqrt{L}$.
\end{corollary}

Assume that there exists some $i\geq0$  and $D_{v}^{0}$ such that
\begin{eqnarray}\label{neq:turing_transition}
\lambda_{i}<k_{c}^{0}<\lambda_{i+1},
~~detM_{\lambda_{i}}>0, ~~detM_{\lambda_{i+1}}>0,
\end{eqnarray}
where $\lambda_{i}$ and $\lambda_{i+1}$ are the eigenvalues of the operator $-\Delta$,  $M_{\lambda_{i}}(k=i)$ is as (\ref{neq:turing_transition}), and
\begin{equation}\label{}
k_{c}^{0}=\frac{a_{11}D_{v}^{0}+a_{22}D_{u}}{2D_{u}D_{v}^{0}}.
\end{equation}
where $D_{v}^{0}$ is defined in \eqref{eq:ex_dv0}. It is easy to check that, $k_{c}^{0}$ is the minimal point of the polynomial $h(\lambda,L)$.

Further denote that
\begin{align}\label{}
   &D_{v}^{\lambda_{i}}=\frac{a_{22}D_{u}\lambda_{i}-det(A)}{D_{u}\lambda_{i}^{2}-a_{11}} ,\\
 &D_{v}^{\lambda_{i+1}}=\frac{a_{22}D_{u}\lambda_{i+1}-det(A)}{D_{u}\lambda_{i+1}^{2}-a_{11}}
 ,\\
 %&D_{v}^{(i)}=\min \{D_{v}^{\lambda_{i}},D_{v}^{\lambda_{i+1}}\},\\
 %&D_{v}^{*}=\inf_{i\ge 1} D_{v}^{(i)}.
 &D_{v}^{*}=\min \{D_{v}^{\lambda_{i}},D_{v}^{\lambda_{i+1}}\}.\label{eq:ex_critical}
\end{align}

Note that since $\lambda_i, i\ge 1$ are bounded below, it is easy to see that $D_{v}^{*}$ is well defined. Theorem \ref{thm:main} can then be improved as follows
\begin{thm} \label{thm:main_cont}
Suppose the conditions in Theorem \ref{thm:main} and (\ref{neq:turing_transition}) holds true, let $D_{u}$ be fixed, then the critical value of the transition for  system \eqref{eq:model} is
$D_{v}^{*}$, i.e, the system \eqref{eq:model} generates the Turing's instability if and only if $D_{v} > D_{v}^{*}$.
\end{thm}
\begin{proof}
Consider the following polynomial
\begin{eqnarray}
h(\lambda)=D_{u}D_{v}\lambda^{2}-(a_{11}D_{v}+
a_{22}D_{u})\lambda + det(A),
\end{eqnarray}
and let $\lambda=\lambda_{i}$ or $\lambda=\lambda_{i+1}$ .
By solving the following equations
\begin{align}\label{}
&h(\lambda_{k})= D_{u}D_{v}\lambda_{i}^{2}-
(a_{11}D_{v}+a_{22}D_{u})\lambda_{i}+det(A)=0,
\\
&h(\lambda_{i+1})= D_{u}D_{v}\lambda_{i+1}^{2}-
(a_{11}D_{v}+a_{22}D_{u})
\lambda_{i+1}+det(A)=0,
\end{align}
we get
\begin{align}\label{}
&D_{v}^{\lambda_{i}}=\frac{a_{22}D_{u}\lambda_{i}
-det(A)}{D_{u}\lambda_{i}^{2}-a_{11}},
\\
&D_{v}^{\lambda_{i+1}}=\frac{a_{22}D_{u}\lambda_{k+1}
-det(A)}{D_{u}\lambda_{i+1}^{2}-a_{11}}.
\end{align}
Besides,
\begin{eqnarray}
\min_{\lambda\in R }\{h(\lambda)\}=det(A)-\frac{(a_{22}D_{u}+
a_{11}D_{v})^{2}}{4D_{u}D_{v}},
\end{eqnarray}
which is bounded below in variable $D_{v}$ for fixed $D_{u}$, so we can take
\begin{equation}\label{}
D_{v}^{(i)}=\inf \{D_{v}^{\lambda_{i}},D_{v}^{\lambda_{i+1}}\}.
\end{equation}
Based on (\ref{eq:turing_det}) and (\ref{neq:turing_transition}), obviously, $D_{v}^{*}$ is the critical point.
\end{proof}

\begin{remark}\label{rmk:calculate}
Here we give the method to find the $D_{v}^{0}$ and $k_{c}^{0}$.
Let
\begin{eqnarray}
\inf_{\lambda\in R }\{h(\lambda)\}=det(A)-\frac{(a_{22}D_{u}+a_{11}D_{v}^0)^{2}}{4D_{u}D_{v}^0}=0,
\end{eqnarray}
that is,
\begin{eqnarray}
a_{11}^{2}(D_{v}^0)^{2}-(4D_{u}det(A)-2a_{11}a_{22}D_{u})D_{v}^0+a_{22}^{2}D_{u}^{2}=0,
\end{eqnarray}
Then we choose the larger positive root, and get
\begin{align}
&k_{c}^{0}=\frac{a_{11}D_{v}^{0}+a_{22}D_{u}}{2D_{u}D_{v}^{0}},\\
&D_{v}^{0}=\frac{q+\sqrt{q^{2}-4a_{11}^{2}a_{22}^{2}D_{u}^{2}}} 
{2a_{11}^{2}}, \label{eq:ex_dv0}
\end{align}
where
\begin{equation*}
q=4D_{u}det(A)-2a_{11}a_{22}D_{u}.
\end{equation*}
\end{remark}

\begin{remark}
Without loss of generality, let $D_{*}=D_{v}^{(1)}$, then the $PES$ condition is shown as follows:

If~~$D_{v}^{\lambda_{i}}\neq D_{v}^{\lambda_{i+1}}$, then
\begin{align}\label{}
 &\beta_{i}^{(2)} \left\{
   \begin{array}{ll}
   <0, ~if ~D_{v}<D_{*},
   \\=0, ~if~ D_{v}=D_{*},
   \\>0, ~if~ D_{v}>D_{*}.
\end{array}
\right.\\
&\beta_{k}^{(j)}(D_{*})<0,~ for ~all~ (k,j)\neq(i,2).
\end{align}

If~~$D_{v}^{\lambda_{i}}= D_{v}^{\lambda_{i+1}}$, then
\begin{align}\label{}
 &\beta_{i}^{(2)}=\beta_{i+1}^{(2)} \left\{
   \begin{array}{ll}
   <0, ~if ~D_{v}<D_{*},
   \\=0, ~if~ D_{v}=D_{*},
   \\>0, ~if~ D_{v}>D_{*}.
\end{array}
\right.\\
&\beta_{k}^{(j)}(D_{*})<0,~ for ~all~ (k,j)\neq(i,2)~and~(i+1,2).
\end{align}
\end{remark}

%%%%%%%%%%%%%%%%%%%%%%%%%%%%%%%%%%%%%%%%%%%%%%%%%%%%%%%%

\subsection{Geometric insights}\label{sec:gemetric}
In this section, we give the geometrical explanation for Turing's instability of system \eqref{eq:model}, which can help us understand the process of Turing's losing stability.
Let
\begin{align}\label{}
&k_{c}=\frac{D}{2D_{u}D_{v}},\\
\label{235}
&h(\lambda,L)=D_{u}D_{v}\lambda^{2}-D\lambda+det(A)\nonumber\\
&=D_{u}D_{v}[(\lambda-k_{c})^{2}-L ],\\
&L_{1}=\min\{L_{\lambda_{i}},L_{\lambda_{i+1}}\},
\end{align}
where
\begin{align}\label{}
  &L_{\lambda_{s}}=(k_{c}-\lambda_{i})^{2}, s=i, i+1,\\
  &\lambda_{i}<k_{c}<\lambda_{i+1},
~~detM_{\lambda_{i}}>0,~~ detM_{\lambda_{i+1}}>0,
\end{align}
 $L$ and $D$ are as in (\ref{eq:ex_transition_L}) and (\ref{eq:ex_turing_D}).
If $k_{c}=\frac{D}{2D_{u}D_{v}}$ is a constant, $\lambda_{i}<k_{c}<\lambda_{i+1}$,
based on (\ref{eq:turing_det}) and (\ref{neq:turing_transition}), obviously,
$L_{1}$ is the critical parameter reflecting the Turing's instability and transition.

\textbf{CASE 1}:$L_{\lambda_{i}} \neq L_{\lambda_{i+1}}$.
\begin{figure}[!htbp]
\begin{center}
\includegraphics[width = .6\textwidth]{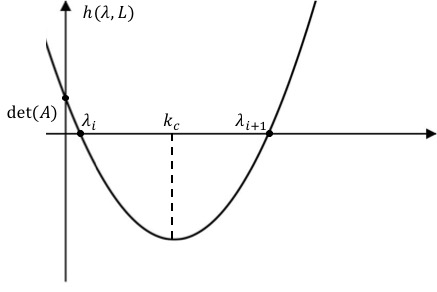}
\caption{A representative graph for Case 1. In this case, $(\lambda_{i}-k_c)^2 = (\lambda_{i+1}-k_c)^2$.}
\label{fig:case1}
\end{center}
\end{figure}

Without loss of generality, taking $L_{1}=L_{\lambda_{k}}$,
which means that $det(M_{\lambda_{k}})<0 $ for $L>L_{1}$, in particular,
\begin{align}\label{241}
 &\beta_{i}^{(2)} \left\{
   \begin{array}{ll}
   <0, ~if ~L<L_{1},
   \\=0, ~if ~L=L_{1},
   \\>0, ~if ~L>L_{1}.
\end{array}
\right.\\
&\beta_{k}^{(j)}(L_{1})<0, ~for ~all (k,j)\neq(i,2),
\end{align}
The transition of \textbf{case 1} is shown in Fig. \ref{fig:case1}.

The curves crossing the fixed point $(0,detA)$ in figure \ref{fig:case1} is determined by (\ref{235}). Based on Theorem \ref{thm:main}, system \eqref{eq:model}
generates Turing's instability if and only if there is a spectral point of $-\Delta$ falling into between the two intersection point of $h(\lambda,L)$ and k-axis. Obviously, Fig. \ref{fig:case1} shows that the spectral point $\lambda_{i}$ of $-\Delta$ is exactly a intersection point of $h(\lambda,L)$ and k-axis.

\textbf{CASE 2}:$L_{\lambda_{i}} = L_{\lambda_{i+1}}$.
\begin{figure}[!htbp]
\begin{center}
\includegraphics[width = .6\textwidth]{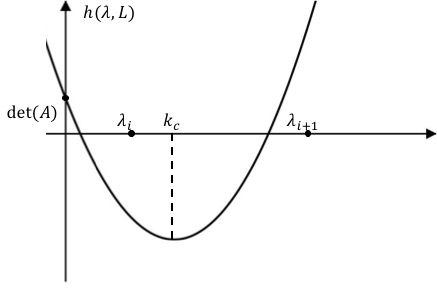}
\caption{A representative graph for Case 2. In this case, $(\lambda_{i}-k_c)^2 \neq (\lambda_{i+1}-k_c)^2$.}
\label{fig:case2}
\end{center}
\end{figure}
In the same way, we can get
\begin{align}\label{243}
 &\beta_{j}^{(2)} \left\{
   \begin{array}{ll}
   <0, if L<L_{1},
   \\=0, if L=L_{1},
   \\>0, if L>L_{1},
\end{array}
\right.
~for~ j=i ~and ~i+1,\\
&\beta_{k}^{(h)}(L_{1})<0, for~ all ~(k,h)\neq(j,2),
\end{align}
The transition of this case is shown in Fig. \ref{fig:case2},
and the $\lambda_{i_{2}}=\lambda_{k}=\lambda_{i_{1}+1}$ are
exactly the two intersection points of $h(\lambda,L)$ and $\lambda$-axis.

In fact, the curve determined by $h(\lambda,L)$ move down as the increasing of the parameter $L(D_{u},D_{v})$, and $L_{1}$ is the critical value of the $L(D_{u},D_{v})$ determined by $h(\lambda,L)$ touching the eigenvalue $\lambda_{i}$ or $ \lambda_{i+1}$.

The two cases mean that Turing's instability and phase transition of system ( \ref{eq:model} ) has only two types, the single real eigenvalue transition for the case that $L_{\lambda_{i}} \neq L_{\lambda_{i+1}}$, and the double real eigenvalue transition for the case that $L_{\lambda_{i}}=L_{\lambda_{i+1}}$.

%%%%%%%%%%%%%%%%%%%%%%%%%%%%%%%%%%%%%%%%%%%%%%%%%%%%%%%%

\subsection{Center manifold reduction of Turing's system}
In addition to Section \ref{sec:cmr}, \textit{center manifold reduction} is also a basic tool to calculate the bifurcated solution in dynamic transition theory, which was established in \cite{Ni1998}. Here we show how this tool can help to derive the bifurcated solution.

  Let $L_{1}=L_{\lambda_{i}}$, the $\xi$ and $\xi_{*}$ be the eigenvectors of matrix $M_{\lambda_{i}}$ and $M_{\lambda_{i}}^{*}$ respectively, and the center manifold function is shown as the follows:
\begin{eqnarray}
\Phi=\Phi(x):\text{span}{e_{\lambda_{i}}}\rightarrow (\text{span}{e_{\lambda_{i}}})^{\perp}
\end{eqnarray}

 For CASE 1, the center manifold reduction system is shown as follows
\begin{eqnarray}\label{eq:cmr_case1}
\frac{dx}{dt}=\beta_{i}^{(2)}x+\frac{<G(x\xi e_{\lambda_{i}}+\Phi(x)),\xi^{*} e_{\lambda_{i}}>_{H}}{<\xi e_{\lambda_{i}},\xi^{*} e_{\lambda_{i}}>_{H}},
\end{eqnarray}
where $\beta_{i}^{(2)}$ is as in (\ref{241}), $e_{\lambda_{i}}$ is a eigenvector of $-\Delta$ corresponding to $\lambda_{i}$, $\xi$ and $\xi^{*}$ are the eigenvectors of the matrix $M_{\lambda_{i}}$ and $ M_{\lambda_{i}}^{*}$ respectively.

For CASE 2, the center manifold reduction system is shown as follows
\begin{align}\label{}
&\frac{dx}{dt}=\beta_{i}^{(2)}x+\frac{<G(x\xi e_{\lambda_{i}}+y \eta e_{\lambda_{i+1}}+\Phi(x,y)),\xi^{*} e_{\lambda_{i}}>_{H}}{<\xi e_{\lambda_{i}},\xi^{*}e_{\lambda_{i}}>_{H}},
 \\
&\frac{dy}{dt}=\beta_{i+1}^{(2)}y+\frac{<G(x\xi e_{\lambda_{i}}+y\eta e_{\lambda_{i+1}}+\Phi(x,y)), \eta^{*} e_{\lambda_{i+1}}>_{H}}{<\eta e_{\lambda_{i+1}},\eta^{*}e_{\lambda_{i+1}}>_{H}}, \label{eq:cmr_case2}
\end{align}
where $\beta_{i}^{(2)}$ and $\beta_{i+1}^{(2)}$ are as (\ref{243}), $e_{\lambda_{i}}$ and $e_{\lambda_{i+1}}$ are the eigenvectors of $-\Delta$ corresponding to $\lambda_{i}$ and $\lambda_{i+1}$ respectively, $\xi$ and $\xi^{*}$ are the eigenvectors of the matrix $M_{\lambda_{i}}$ and
$ M_{\lambda_{i}}^{*}$ respectively, $\eta$ and $\eta^{*}$ are the eigenvectors (a $2\times1$ vector of functions) of the matrix $M_{\lambda_{i+1}}$ and $M_{\lambda_{i+1}}^{*}$ respectively, and $\Phi(x,y):span\{e_{\lambda_{i}},e_{\lambda_{i+1}}\}\rightarrow (span\{e_{\lambda_{i}},e_{\lambda_{i+1}}\})^{\perp}$
is the center manifold function.

Therefore, we basically reduced an infinite dimensional system \eqref{eq:model} to an one or two-dimensional dynamical system, depending on the $i-$th eigenspace of \eqref{eq:model}. \eqref{eq:cmr_case1}-\eqref{eq:cmr_case2} total determines all the bifurcated solutions of \eqref{eq:model} by expressing the bifurcated solutions locally using a center manifold function defined in Theorem \ref{thm:cmr_infinite}.
\newline

%%%%%%%%%%%%%%%%%%%%%%%%%%%%%%%%%%%%%%%%%%%%%%%%%%%%%%%%
%%%%%%%%%%%%%%%%%%%%%%%%%%%%%%%%%%%%%%%%%%%%%%%%%%%%%%%%

\section{Application to the Schnakenberg system}\label{sec:5}
The Schnakenberg system is a well-studied reaction-diffusion systems. It is a classical example of non-equilibrium thermodynamics resulting in the establishment of a nonlinear chemical oscillator \cite{Schnakenberg1979}. It has also been used to model the spatial distribution of a morphogen, e.g., the distribution of calcium in the tips and whorl in Acetabularia \cite{Goodwin1985}. As reviewed at the beginning of this paper, morphogen-based mechanisms have been widely proposed for tissue patterning, but only recently have there been sufficient experimental data and adequate modeling for us to begin to understand how various morphogens interact with cells and emergent patterns \cite{Gurdon2001,Arafa2012}.

Denote $X$, $A$, $B$ and $Y$ to be four different chemicals, Schnakenberg considered the following chemical reaction
\bea\label{eq:sch_reaction}
X \xrightleftharpoons[k_1]{k_{-1}} A,\quad B \overset{k_2}\rightarrow Y,\quad 2X + Y \overset{k_3}\rightarrow X
\eea
%%%%%%%%%%%%%%%%%%%%%%%%%%%%%%%%%%%%%%%%%%%%%%%%%%%%%%%%

\subsection{Mathematical form of the Schnakenberg system }
If concentrations $A$ and $B$ are approximately constants (e.g. $A$ and $B$ are abundant in the system), after proper nondimensionalization and impose Dirichlet/Neumann boundary condition on a bounded domain $\Omega\subset \mathbb{R}^n$, then the mathematical form of the Schnakenberg model is given by
 \begin{eqnarray}\label{eq:sch}
\left\{
   \begin{array}{ll}
   \frac{\partial u}{\partial t}=\Delta u +r(a- u+u^{2}v),
   \\\frac{\partial v}{\partial t}=d \Delta v +r(b- u^{2}v) .
   %\\\frac{\partial u}{\partial n}|_{\partial \Omega}=0 \text{~or~} u|_{\partial \Omega}=0.
\end{array}
\right.
\end{eqnarray}
 where $u$ and $ v $ are all the concentrations of the chemicals, $u$ is activator, and $v$ is inhibitor.
Obviously, the steady state of \eqref{eq:sch} is as follows
\begin{equation}\label{eq:sch_sss}
 (u_{0},v_{0})=(a+b,\frac{b}{(a+b)^{2}}).
\end{equation}

Make the following substitution
\begin{equation*}
  u=a+b+u_{1}, ~v=\frac{b}{(a+b)^{2}}+v_{1},
\end{equation*}
then equation \eqref{eq:sch} can be rewritten as
 \begin{eqnarray}\label{eq:model_sch}
\left\{
   \begin{array}{ll}
   \frac{\partial u_{1}}{\partial t}=\Delta u_{1} +r( \frac{b-a}{a+b}u_{1}+(a+b)^{2}v_{1}+
   2(a+b)u_{1}v_{1}-\frac{b}{(a+b)^{2}}u_{1}^{2}+u_{1}^{2}v_{1}),
   \\\frac{\partial v_{1}}{\partial t}=d\Delta v_{1} -r( \frac{2b}{a+b}u_{1}+(a+b)^{2}v_{1}+
   2(a+b)u_{1}v_{1}-\frac{b}{(a+b)^{2}}u_{1}^{2}+u_{1}^{2}v_{1}) .
   %\\\frac{\partial v}{\partial n}|_{\partial \Omega}=0.
\end{array}
\right.
\end{eqnarray}

By (\ref{205})-(\ref{eq:cmr_cond}), equation (\ref{eq:model_sch}) is equivalents to following operator equation
\begin{eqnarray}\label{eq:sch_abstract}
\left\{
   \begin{array}{ll}
   \frac{\partial w}{\partial t}=L_{\lambda}w + G(w,\lambda),
   \\v(0)=\varphi,
\end{array}
\right.
\end{eqnarray}
where
\begin{align}\label{305}
&L_{\lambda}=A_{\lambda}+B_{\lambda},G(w,\lambda):H_{1}\rightarrow H,\\
&w=(u_{1},v_{1})^{T},
\\ &B_{\lambda}u=(\Delta u_{1},d\Delta v_{1}),\\
 &A_{\lambda}v=
\left(
\begin{array}{cc}
  r\frac{b-a}{a+b}& r(a+b)^{2}\\
  -r\frac{2b}{a+b}& -r(a+b)^{2}
\end{array}
\right)w,
\\& G(w,\lambda)=(g_{1}(w,r),g_{2}(w,r))
\\&g_{1}(w,r)=r(2(a+b)u_{1}v_{1}
-\frac{b}{(a+b)^{2}}u_{1}^{2}+u_{1}^{2}v_{1}),
\\&g_{2}(w,r)=-r(2(a+b)u_{1}v_{1}-
\frac{b}{(a+b)^{2}}u_{1}^{2}+u_{1}^{2}v_{1}),
\\&\lambda=(1,\alpha,r)\in R_{3}^{+}=\{ (x_{1},x_{2},x_{3})\in R^{3}:x_{i}\geq 0,i=1,2,3 \}\nonumber.
\end{align}
and also
\begin{align}\label{}
 &a_{11}= r\frac{b-a}{a+b},a_{12}= r(a+b)^{2}, \\
 &a_{21}= -r\frac{2b}{a+b},a_{22}= -r(a+b)^{2}.
\end{align}

\subsection{A necessary and sufficient condition for Turing's instability }
For system (\ref{eq:sch}), based on (\ref{eq:ex_transition_L}) and (\ref{eq:ex_turing_D}), we have
\begin{align}
& D=rd\frac{b-a}{a+b}-r(a+b)^{2},\\
 &det(A)=r^{2}(a+b)^{2},\\
 &L=-\frac{r^{2}(a+b)^{2}}{d} +\frac{D^{2}}{4d^{2}},
\end{align}
Let
\begin{equation}\label{}
 \beta_{\pm}=\frac{D\pm\sqrt{D^{2}-4d ~det(A)}}{2d},
\end{equation}
based on Theorem (\ref{thm:main}) , we get the following conclusion
\begin{thm}\label{thm:sch_1}
Assume $\frac{b-a}{a+b}-(a+b)^{2}<0$\\
(1) if
$-\frac{\sqrt{4d}}{r(a+b)}<d\frac{b-a}{a+b}-(a+b)^{2}<0$, i.e.
\bea
\max{(d,1)}(b-a) < (a+b)^3 < (d+\frac{\sqrt{4d}}{r})(b-a)
\eea
hen system (\ref{eq:sch}) is Turing stable.\\
(2) if
$d\frac{b-a}{a+b}-(a+b)^{2}>\frac{\sqrt{4d}}{r(a+b)}$ or $d\frac{b-a}{a+b}-(a+b)^{2}<-\frac{\sqrt{4d}}{r(a+b)}$, i.e.
\bea
(b-a) < (a+b)^3 < (d-\frac{\sqrt{4d}}{r})(b-a) 
\eea
\bea
\text{or~}\max{(1,d+\frac{\sqrt{4d}}{r})}(b-a) < (a+b)^3 
\eea
then system (\ref{eq:sch}) generates Tuing's instability if and only if there exists an eigenvalue $\lambda_{k_{0}}$ of $-\Delta$ such that $\beta_{-}<\lambda_{k_{0}}<\beta_{+}$.
\end{thm}

%%%%%%%%%%%%%%%%%%%%%%%%%%%%%%%%%%%%%%%%%%%%%%%%%%%%%%%%

\subsection{The critical parameter of the Schnakenberg system}
In fact, $d$ is an adjustable parameter for Schnakenberg system. That means that the critical parameter is determined by $d$ . In the following, we will give the critical parameter $d_{0}$.

Due to the method introduced in Section \ref{sec:cmr}, by directed calculation we get
\begin{align}\label{}
 & h(\lambda)=d\lambda^{2}-r(\frac{b-a}{b+a}d-
 (a+b)^{2})\lambda+r^{2}(a+b)^{2},\\
  &k_{c}=\frac{r(b-a)d-r(a+b)^{3}}{2d(a+b)}.
\end{align}
Let
\begin{equation}\label{}
D_{v}^{0}=\frac{(4b^{2}+4ab)(a+b)^{2}+
(a+b)^{2}\sqrt{(4b^{2}+4ab)^{2}+a^{2}-b^{2}}}{(b-a)^{2}},
\end{equation}
\begin{align}\label{}
&k_{c}^{0}=r\frac{(b-a)D_{v}^{0}-(b+a)^{3}}{2D_{v}^{0}(b+a)},\\
&d_{0}^{\lambda_{k}}=\frac{-r^{2}(a+b)^{3}-r(a+b)^{3}\lambda_{k}}
{(a+b)\lambda_{k}^{2}-r(b-a)},\\
&d_{0}^{\lambda_{k+1}}=\frac{-r^{2}(a+b)^{3}
-r(a+b)^{3}\lambda_{k+1}}{(a+b)\lambda_{k+1}^{2}-r(b-a)},\\
&\label{323}d_{0}=\min\{ d_{0}^{\lambda_{k}} , d_{0}^{\lambda_{k+1}}\},
\end{align}
where $\lambda_{k}$ and $\lambda_{k+1}$ are the eigenvalues of the operator $-\Delta$ such that
\begin{eqnarray}\label{eq:sch_ass}
\lambda_{k}<k_{c}^{0}<\lambda_{k+1},~ h(\lambda_{k})>0, ~h(\lambda_{k+1})>0,
\end{eqnarray}

Based on Theorem (\ref{thm:main_cont}) in section 2, then we get the following corollary.
\begin{corollary} Let $d_{0}$ be as in (\ref{323}), then $d_{0}$ is the critical parameter reflecting Turing's instability and phase transition, i.e, if $d>d_{0}$, then Turing's instability appears and dynamic transition occurs.
\end{corollary}

%%%%%%%%%%%%%%%%%%%%%%%%%%%%%%%%%%%%%%%%%%%%%%%%%%%%%%%%

\subsection{Phase Transition of the Schnakenberg system  }
\subsubsection{Single real eigenvalue transition of the Schnakenberg system}
~~~~Based on the dynamic bifurcation theory in \cite{Ni1998}, without loss of generality, assume that $\lambda_{i_{1}}=\lambda_{1}$, then the center manifold reduction equation for system (\ref{eq:sch}) is given by
    \begin{equation}\label{322}
    \frac{dx}{dt}=\beta_{\lambda_{1}}^{(2)}x+\frac{1}{\langle \xi e_{\lambda_{1}},\xi ^{*}e_{\lambda_{1}}\rangle}_{H}
    \langle G(x\xi e_{\lambda_{1}}+\Phi(x)),\xi ^{*}e_{\lambda_{1}}\rangle_{H}
    \end{equation}
where
    \begin{equation}\label{}
    M_{\lambda_{1}}\xi=0, M_{\lambda_{1}}^{*}\xi^{*}=0,
    \end{equation}
\begin{equation}\label{}
M_{\lambda_{1}}=\left(
\begin{array}{cc}
  r\frac{b-a}{a+b}-\lambda_{1}^{2}& r(a+b)^{2}\\
  -r\frac{2b}{a+b}& -r(a+b)^{2}-\lambda_{1}^{2}
\end{array}
\right),M_{\lambda_{1}}^{*}=M_{\lambda_{1}}^{T}.
\end{equation}
\begin{equation}\label{}
  \xi =(\xi_{1} ,\xi_{2} )=(a+b,\frac{(a+b)\lambda_{1}^{2}-(b-a)r}{r(a+b)^{2}}),
\end{equation}
\begin{equation}\label{}
  \xi^{*} =(\xi_{1}^{*} ,\xi_{2}^{*} )=(\frac{2br(a+b)}{(b-a)r-(a+b)\lambda_{1}^{2}},a+b).
\end{equation}
Note
\begin{equation*}
  h=\bigg(\frac{2br(a+b)^{2}}{(b-a)r-(a+b)\lambda_{1}^{2}}
  +\frac{(a+b)\lambda_{1}^{2}-(b-a)r}{r(a+b)}\bigg)\int_{\Omega} e_{\lambda_{1}}^{2}dx,
\end{equation*}
thus we get the second order term and the third order term as follows.
\begin{align}\label{}
&G_{2}=r(-Hu^{2}+Muv,Hu^{2}-Muv),\\
&G_{3}=r(u^{2}v,-u^{2}v),\\
&H=\frac{b}{(a+b)^{2}},\\
&M=2(a+b).
\end{align}

We can also get
\bea
% \nonumber to remove numbering (before each equation)
  &\frac{1}{h}\langle G_{2}(x\xi_{1}e_{\lambda_{1}},x\xi_{2}e_{\lambda_{1}}),\xi^{*}e_{\lambda_{1}}\rangle_{H} \nonumber\\&=\frac{r}{h}\bigg( (\xi_{2}^{*}-\xi_{1}^{*})H \xi_{1}^{2}\int_{\Omega} e_{\lambda_{1}}^{3}dx+(\xi_{2}^{*}-\xi_{1}^{*})\xi_{1}
  \xi_{2}M\int_{\Omega} e_{\lambda_{1}}^{3}dx\bigg)x^{2}
   \nonumber\\&=Px^{2}.
\eea

%\begin{eqnarray}
% \nonumber to remove numbering (before each equation)
  %&\frac{1}{h}\langle G_{3}(x\xi_{1}e_{\lambda_{1}},x\xi_{2}e_{\lambda_{1}}),\xi^{*}e_{\lambda_{1}}\rangle_{H} \nonumber\\&=\frac{r}{h}\bigg( (\xi_{2}^{*}-\xi_{1}^{*})\xi_{1}^{2}\xi_{2}\int_{\Omega} e_{\lambda_{1}}^{4}dx\bigg)x^{3}
   %\nonumber\\&=Vx^{3}
%\end{eqnarray}

Hence, the center manifold reduction equation (\ref{322}) can be written as
\begin{equation}\label{}
\frac{dx}{dt}=\beta_{\lambda_{1}}^{(2)}x+Px^{2}.
\end{equation}

Suppose
\begin{equation}\label{333}
\int_{\Omega} e_{\lambda_{1}}^{3}dx\neq0.
\end{equation}

We have the following Theorem.
\begin{thm}\label{thm:sch:single}
Let $\lambda_{1}^{2}\neq r$ and $ \int e_{\lambda_{1}}^{3}dx\neq 0$, then the system (\ref{eq:sch}) has a transition at $d=d_{0}$, which is mixed transition. In particular, the system bifurcates on each side of $d=d_{0}$ to a unique branch $w^{d}$ of steady state solutions, such that the following assertions hold true:\\
(1) When $d<d_{0}$, the bifurcated solution $w^{d}$ is a saddle, and the stable manifold $\Gamma_{d}^{1}$ of $w^{d}$ separates the space H into two disjoint open sets $ U^{d}$ and$ V^{d}$, such that $ v = 0 \in V^{d}$ is an attractor, and the orbits of (3.24 ) in  $ u^{d}$are far from
$v = 0$.\\
(2) When $d>d_{0}$, the stable manifold $\Gamma_{d}^{0}$ of $v = 0 $ separates the neighborhood O
of $u = 0$ into two disjoint open sets $O_{1}^{d}$ and $O_{2}^{d}$, such that the transition is jump in $O_{1}^{0}$, and is continuous in $O_{2}^{d}$. The bifurcated solution $w^{d}\in O_{2}^{d}$is an attractor such that for any $\phi\in O_{2}^{d}$,we have
\begin{equation}\label{}
\lim_{t\rightarrow\infty}\|w(t,\phi)- w^{d} \|_{H}=0
\end{equation}
where $w(t,\phi)$ is the solution of (\ref{eq:sch}) with $w(0,\phi)=\phi.$\\
(3) The bifurcated solution $w^{d}$ can be expressed as
\begin{equation}\label{}
w^{d}=-\frac{\beta_{\lambda_{1}}^{(2)}}{P}\xi e_{\lambda_{1}}(x)
\end{equation}
$\xi$ and $\beta_{\lambda_{1}}^{(2)}$ are shown above.
\end{thm}
\begin{proof}
The reduction system (\ref{322}) equivalents to
\begin{equation}\label{}
\frac{dx}{dt}=\beta_{\lambda_{1}}^{(2)}x+Px^{2}.
\end{equation}
Due to the PES condition as follows:
\begin{align}\label{340}
&\beta_{\lambda_{1}}^{(2)} \left\{
   \begin{array}{ll}
   <0, ~if~ d<d_{0};
   \\=0, ~if~ d=d_{0};
   \\>0, ~if~ d>d_{0};
\end{array}
\right.\\&
\beta_{k}^{(j)}(d_{0})<0, ~for ~all (k,j)\neq(\lambda_{1}2).
\end{align}

Therefore, all the above conclusions hold true.

The transition behavior of \ref{340} is shown in Fig. \ref{fig:dynamic_1}.
%\begin{figure}[!hbtp]
%  \centering
%  % Requires \usepackage{graphicx}
%  \includegraphics[width=.6\textwidth]{Imgs/dynamicChange1.jpg}\\
%  \caption{}\label{fig:dynamic_1}
%\end{figure}
\begin{figure}[!hbtp]
\centering
\begin{tikzpicture}[>=stealth]
  \coordinate (o) at (0,0);
  \draw node (a) at (-2,0) {};
  %\draw node[vertex] (b) at (-1,0) {};
  \draw node (b) at (-1,0) {};
  \draw node (c) at (0,0) {};
  \draw node (d) at (1,0) {};
  \draw[thick, ->] (b) -- (a); \draw[thick, ->] (-1,1) -- (b); \draw[thick, ->] (b) -- (c); \draw[thick, ->] (-1,-1) -- (b);
  \draw[thick, ->] (d) -- (c); \draw[thick, ->] (0,1) -- (c); \draw[thick, ->] (0,-1) -- (c);
  \draw[thick, ->] (-1.2,1) to [bend left=45] (-1.8,0.2); \draw[thick, ->] (-0.8,1) to [bend right=45] (-.2,0.2);
  \draw[thick, ->] (-1.2,-1) to [bend right=45] (-1.8,-0.2); \draw[thick, ->] (-0.8,-1) to [bend left=45] (-.2,-0.2);
  \node[above] at (.1,0) {$o$}; \node[right] at (-1,.2) {$v^{d}$};
  \node[below] at (1,-.3) {$U_{2}^{d}$}; \node[below] at (-2,-.5) {$U_{1}^{d}$};
  \node[below] at (-1,-1.3) {$d<d_0$};
  \end{tikzpicture}\hspace{30pt}
  \begin{tikzpicture}[>=stealth]
  \coordinate (o) at (0,0);
  \draw node (a) at (-2,0) {};
  \draw node (b) at (-1,0) {};
  \draw node (c) at (0,0) {};
  \draw node (d) at (1,0) {};
  \draw[thick, ->] (a) -- (b); \draw[thick, ->] (-1,1) -- (b); \draw[thick, ->] (c) -- (b); \draw[thick, ->] (-1,-1) -- (b);
  \draw[thick, ->] (c) -- (d); \draw[thick, ->] (0,1) -- (c);\draw[thick, ->] (0,-1) -- (c);
  \draw[thick, ->] (-.2,1) to [bend left=45] (-.8,0.2); \draw[thick, ->] (.2,1) to [bend right=45] (.8,0.2);
  \draw[thick, ->] (-.2,-1) to [bend right=45] (-.8,-0.2); \draw[thick, ->] (.2,-1) to [bend left=45] (.8,-0.2);
  \node[above] at (-.1,0) {$o$}; \node[left] at (-1,.2) {$v^{d}$};
  \node[below] at (1,-.3) {$O_{1}^{d}$}; \node[below] at (-2,-.5) {$O_{2}^{d}$};
  \node[below] at (-1,-1.3) {$d>d_0$};
    \end{tikzpicture}
  \caption{Transition diagram for Theorem \ref{thm:sch:single}.}\label{fig:dynamic_1}
\end{figure}
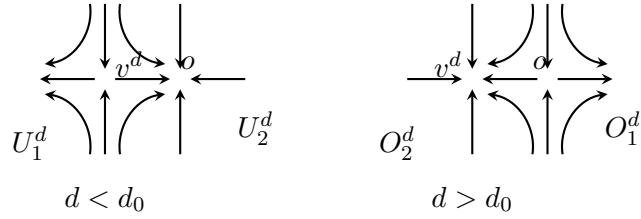
\end{proof}

If (\ref{333}) is not true, i.e,
\begin{equation}\label{339}
\int_{\Omega} e_{\lambda_{1}}^{3}dx=0.
\end{equation}

We introduce the following parameter
\begin{equation}\label{s343}
Q=\frac{1}{h}(\xi_{1}^{*}-\xi_{2}^{*})Mr\xi_{1}\int_{\Omega}\psi_{2}e_{\lambda_{1}}^{2}{\rm d}x-2(\xi_{1}^{*}-\xi_{2}^{*})rH\xi_{1}\int_{\Omega}\psi_{1}e_{\lambda_{1}}^{2}{\rm d}x
\end{equation}
Where$ \psi=(\psi_{1},\psi_{2})$ satisfied the following equation
\begin{align}\label{341}
  &\Delta\psi_{1}+r\frac{b-a}{b+a}\psi_{1}+r(a+b)^{2}\psi_{2}=-rH\xi_{1}^{2}e_{\lambda_{1}}^{2} +rM\xi_{1}\xi_{2}e_{\lambda_{1}}^{2}, \\
  \label{342}&\Delta\psi_{1}-r\frac{2b}{b+a}\psi_{1}-r(a+b)^{2}\psi_{2}=rH\xi_{1}^{2}e_{\lambda_{1}}^{2} -rM\xi_{1}\xi_{2}e_{\lambda_{1}}^{2}.
 % \label{343}&\frac{\partial\psi}{\partial n}=0  \text{~or~} \phi=0 ~~on~~\partial\Omega.
\end{align}

By the Fredholm Alternative Theorem, under the condition (\ref{339}), the equations (\ref{341})-(\ref{342})
have a unique solutions.
Moreover, we have the following conclusion.

\begin{thm} \label{thm:solution}
Let $\lambda_{1}^{2}\neq r$, $ \int e_{\lambda_{1}}^{3}dx= 0$, and $Q$ be the number given by (\ref{s343}), then the system (\ref{eq:sch}) has a transition at $d=d_{0}$, and the transition of (\ref{eq:sch}) at $d=d_{0}$ is continuous if $Q < 0$, jump if $Q > 0$.
The following assertions hold true:\\
(1)If $Q > 0$,(\ref{eq:sch}) has no bifurcation when $d>d_{0}$, and has exact two bifurcated solutions $w_{+}^{d}$and $w_{-}^{d}$ which are saddles when $d<d_{0}$. Moreover, the stable manifolds $\Gamma_{+}^{d}$ and $\Gamma_{-}^{d}$ of the two bifurcated solutions divide H into three disjoint open sets $U_{+}^{d}$, $U_{0}^{d}$, $U_{-}^{d}$such that $v =0\in U_{0}^{d}$is an attractor, and the orbits in $U_{\pm}^{d}$ are far from $v = 0$.
\\
(2)If $Q < 0$,(\ref{eq:sch}) has no bifurcation when $d< d_{0}$, and has exact two bifurcated solutions $w_{+}^{d}$ and $w_{-}^{d}$ when $d>d_{0}$ , which are attractors. In addition, there is a neighborhood $O \subseteq H $of $w = 0$, such that the stable manifold $\gamma$ of $w = 0 $ divides $O$ into two disjoint open sets $U_{+}^{d}$ and $U_{-}^{d}$ such that $w_{+}^{d} \in U_{+}^{d}$, $w_{-}^{d} \in U_{-}^{d}$, and$w_{\pm}^{d}$ attracts $U_{\pm}^{d}$;
\\
(3)The bifurcated solutions $w_{\pm}^{d}$ can be expressed as
\begin{equation}\label{}
w_{\pm}^{d}=\pm(-\frac{\beta_{\lambda_{1}}^{(2)}}{Q})^{\frac{1}{2}}\xi e_{\lambda_{1}}(x)
\end{equation}
$\xi$ and $\beta_{\lambda_{1}}^{(2)}$ is shown above.
\end{thm}
\begin{proof}
To prove this result, we need to calculate the center manifold function $\Phi(x)$. By the procedures in Section 3 of \cite{Ma2005} and Appendix A of \cite{Ma2015},
$\Phi(x)$ satisfies:
\begin{equation}\label{345}
L_{\lambda}\Phi=-P_{2}G(x\xi e_{\lambda_{1}})
\end{equation}
where $ P_{2}: H\rightarrow E_{2} $ is the canonical projection, $L_{\lambda}$ is as in (\ref{305}), $e_{\lambda_{1}}$ and $e_{\lambda_{1}}^{*}$ are given by (\ref{339}), and
\begin{equation}\label{}
E_{2}=\{v\in H:(v,e_{\lambda_{1}}^{*})=0\}
\end{equation}
we can see that
\begin{eqnarray}
% \nonumber to remove numbering (before each equation)
  G(x\xi e_{\lambda_{1}})=( -rH\xi_{1}^{2}e_{\lambda_{1}}^{2}x^{2}+rM\xi_{1}\xi_{2}e_{\lambda_{1}}^{2}x^{2}, \\
 rH\xi_{1}^{2}e_{\lambda_{1}}^{2}x^{2}-rM\xi_{1}\xi_{2}e_{\lambda_{1}}^{2}x^{2}).
\end{eqnarray}
Let
\begin{equation}\label{349}
\Phi=\psi x^{2}+0(2),\psi =(\psi_{1},\psi_{2})\in H.
\end{equation}
By (\ref{339}), $(e_{\lambda_{1}}^{2},-e_{\lambda_{1}}^{2})\in E_{2}$. Hence, it follows from (\ref{345}) and (\ref{349}) that\\
\begin{equation}\label{}
L_{\lambda}\psi=(-r^{2}H\xi_{1}^{2}x^{2}+rM\xi_{1}\xi_{2}x^{2})(e_{\lambda_{1}}^{2},-e_{\lambda_{1}}^{2})
\end{equation}
By direct calculation, we obtain the following
\begin{align}\label{}
   &<G(x\xi e_{\lambda_{1}}+\psi x_{2}),\xi^{*}e_{\lambda_{1}}>_{H}\\
  &=((Mr\xi_{1}\psi_{2}e_{\lambda_{1}}-2rH\xi_{1}\psi_{1})x^{3}+0(3)\\
 &-(Mr\xi_{1}\psi_{2}e_{\lambda_{1}}-2rH\xi_{1}\psi_{1})x^{3}+0(3))=Q x^{3},
\end{align}
where
\begin{equation}\label{}
Q=\frac{1}{h}(\xi_{1}^{*}-\xi_{2}^{*})Mr\xi_{1}\int_{\Omega}\psi_{2}e_{\lambda_{1}}^{2}{\rm d}x-2(\xi_{1}^{*}-\xi_{2}^{*})rH\xi_{1}\int_{\Omega}\psi_{1}e_{\lambda_{1}}^{2}{\rm d}x
\end{equation}

In the end, we get the center manifold reduction system as follows:
\begin{equation}\label{}
\frac{dx}{dt}=\beta_{\lambda_{1}}^{(2)}x+Qx^{3}.
\end{equation}
Whose steady state solutions are
\begin{equation}\label{}
x_{\pm}=\pm(-\frac{\beta_{\lambda_{1}}^{(2)}}{Q})^{\frac{1}{2}}
\end{equation}
Therefore, following procedures found in Section 3 of \cite{Ma2005} and Appendix A of \cite{Ma2015}, all the above conclusions hold true.

Fig. \ref{fig:dynamic_2} and \ref{fig:dynamic_3} are corresponding transition diagrams for the cases when $Q>0$ and $Q<0$ respectively.
%\begin{figure}[!hbtp]
%  \centering
%  \includegraphics[width=.6\textwidth]{Imgs/dynamicChange2.jpg}\\
%  \caption{Transition diagram when $Q>0$.}\label{fig:dynamic_2}
%\end{figure}

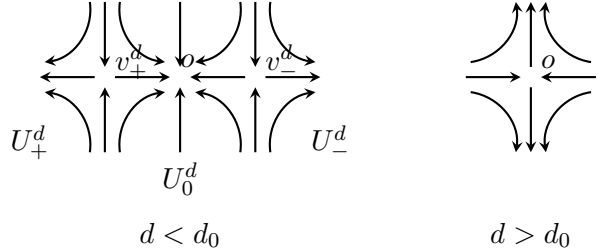
\begin{figure}[!hbtp] \label{fig:dynamic_2}
  \centering
%  \begin{tikzpicture}
%      \draw[<->]
%        (0,4) node[above]{$\tilde{g}_1$} -- (0,0) -- (6,0) node[right]{$v$} ;
%      \draw[very thick,line cap=round]
%        (.8,4)
%        .. controls (1,0) and (1.5,1) .. (2.5,2)
%        node[pos=.56,circle,fill,minimum size=4,inner sep=0]{}
%        node[pos=.56,below]{Liquid}
%        .. controls (3,2.5) and (3.2,2.5) .. (3.5,2)
%        .. controls (3.8,1.5) .. (6,3)
%        node[pos=.37,circle,fill,minimum size=4,inner sep=0]{}
%        node[pos=.37,below]{Gas} ;
%  \end{tikzpicture}
  \begin{tikzpicture}[>=stealth]
  \coordinate (o) at (0,0);
  \draw node (a) at (-2,0) {};
  %\draw node[vertex] (b) at (-1,0) {};
  \draw node (b) at (-1,0) {};
  \draw node (c) at (0,0) {};
  \draw node (d) at (1,0) {};
  \draw node (e) at (2,0) {};
  \draw[thick, ->] (b) -- (a); \draw[thick, ->] (-1,1) -- (b); \draw[thick, ->] (b) -- (c); \draw[thick, ->] (-1,-1) -- (b);
  \draw[thick, ->] (d) -- (e); \draw[thick, ->] (1,1) -- (d); \draw[thick, ->] (d) -- (c); \draw[thick, ->] (1,-1) -- (d);
  \draw[thick, ->] (0,-1) -- (c); \draw[thick, ->] (0,1) -- (c); 
  \draw[thick, ->] (-1.2,1) to [bend left=45] (-1.8,0.2); \draw[thick, ->] (-0.8,1) to [bend right=45] (-.2,0.2);
  \draw[thick, ->] (-1.2,-1) to [bend right=45] (-1.8,-0.2); \draw[thick, ->] (-0.8,-1) to [bend left=45] (-.2,-0.2);
  \draw[thick, ->] (1.2,1) to [bend right=45] (1.8,0.2); \draw[thick, ->] (0.8,1) to [bend left=45] (.2,0.2);
  \draw[thick, ->] (1.2,-1) to [bend left=45] (1.8,-0.2); \draw[thick, ->] (0.8,-1) to [bend right=45] (.2,-0.2);
  \node[above] at (.1,0) {$o$}; \node[right] at (-1,.2) {$v_{+}^{d}$}; \node[right] at (1,.2) {$v_{-}^{d}$};
  \node[below] at (0,-1) {$U_{0}^{d}$}; \node[below] at (-2,-.5) {$U_{+}^{d}$}; \node[below] at (2,-.5) {$U_{-}^{d}$};
  \node[below] at (0,-1.8) {$d<d_0$};
    \end{tikzpicture}\hspace{30pt}
  \begin{tikzpicture}[>=stealth]
  \coordinate (o) at (0,0);
  \draw node (b) at (-1,0) {};
  \draw node (c) at (0,0) {};
  \draw node (d) at (1,0) {};
  \draw[thick, ->] (c) -- (0,1); \draw[thick, ->] (c) -- (0,-1);
  \draw[thick, ->] (b) -- (c); \draw[thick, ->] (d) -- (c); 
  \draw[thick, ->] (-.8,.2) to [bend right=45] (-.2,1); \draw[thick, ->] (0.8,.2) to [bend left=45] (.2,1);
  \draw[thick, ->] (-.8,-.2) to [bend left=45] (-.2,-1); \draw[thick, ->] (0.8,-.2) to [bend right=45] (.2,-1);
  \node[right] at (0,.2) {$o$};
  \node[below] at (0,-1.8) {$d>d_0$};
  \end{tikzpicture}
  \caption{Transition diagram when $Q>0$.}\label{fig:dynamic_2}
  \end{figure}

%\begin{figure}[!hbtp]
%  \centering
%  \includegraphics[width=.6\textwidth]{Imgs/dynamicChange3.jpg}\\
%  \caption{Transition diagram when $Q<0$.}\label{fig:dynamic_3}
%\end{figure}

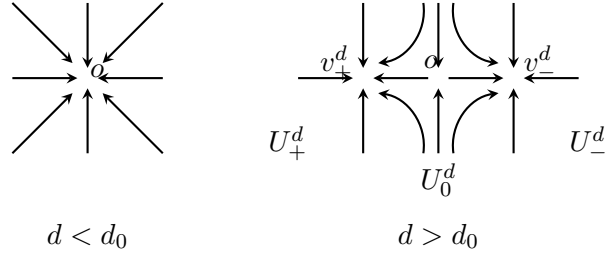
\begin{figure}[!hbtp]
  \centering
  \begin{tikzpicture}[>=stealth]
  \coordinate (o) at (0,0);
  \draw node (b) at (-1,0) {};
  \draw node (c) at (0,0) {};
  \draw node (d) at (1,0) {};
  \draw[thick, ->] (0,1) -- (c); \draw[thick, ->] (0,-1) -- (c);
  \draw[thick, ->] (1,0) -- (c); \draw[thick, ->] (-1,0) -- (c);
  \draw[thick, ->] (1,1) -- (.2,.2); \draw[thick, ->] (-1,-1) -- (-.2,-.2);
  \draw[thick, ->] (1,-1) -- (.2,-.2); \draw[thick, ->] (-1,1) -- (-.2,.2);
  \node[right] at (-.1,.1) {$o$};
  \node[below] at (0,-1.8) {$d<d_0$};
  \end{tikzpicture}\hspace{30pt}
  \begin{tikzpicture}[>=stealth]
  \coordinate (o) at (0,0);
  \draw node (a) at (-2,0) {};
  \draw node (b) at (-1,0) {};
  \draw node (c) at (0,0) {};
  \draw node (d) at (1,0) {};
  \draw node (e) at (2,0) {};
  \draw[thick, ->] (a) -- (b); \draw[thick, ->] (-1,1) -- (b); \draw[thick, ->] (c) -- (b); \draw[thick, ->] (-1,-1) -- (b);
  \draw[thick, ->] (e) -- (d); \draw[thick, ->] (1,1) -- (d); \draw[thick, ->] (c) -- (d); \draw[thick, ->] (1,-1) -- (d);
  \draw[thick, ->] (0,-1) -- (c); \draw[thick, ->] (0,1) -- (c); 
  \draw[thick, ->] (-.2,1) to [bend left=45] (-.8,0.2); \draw[thick, ->] (.2,1) to [bend right=45] (.8,0.2);
  \draw[thick, ->] (-.2,-1) to [bend right=45] (-.8,-0.2); \draw[thick, ->] (.2,-1) to [bend left=45] (.8,-0.2);
  \node[above] at (-.1,0) {$o$}; \node[left] at (-1,.2) {$v_{+}^{d}$}; \node[right] at (1,.2) {$v_{-}^{d}$};
  \node[below] at (0,-1) {$U_{0}^{d}$}; \node[below] at (-2,-.5) {$U_{+}^{d}$}; \node[below] at (2,-.5) {$U_{-}^{d}$};
  \node[below] at (0,-1.8) {$d>d_0$};
    \end{tikzpicture}
  \caption{Transition diagram when $Q<0$.}\label{fig:dynamic_3}
\end{figure}

\end{proof}
~~~If  $\lambda_{1}^{2}= r$, for the case that $P=Q=0$, we have the following conclusion.
\begin{thm}
The attractor $ v=0$ is stable if $d<d_{0}$, and unstable if $d<d_{0}$ for system (\ref{eq:sch}), there is no solution bifurcated from the critical parameter $d=d_{0}$.
\end{thm}
\begin{proof}
The condition $\lambda_{1}^{2}= r$ is equavalent to $\xi_{1}^{*}=\xi_{2}^{*} $, so $P=Q=0$ . The conclusion can be simply obtained from the following reduction system
\begin{equation}\label{}
\frac{dx}{dt}=\beta_{\lambda_{1}}^{(2)}x
\end{equation}
The proof is complete.
\end{proof}
%Note
%\begin{equation*}
 % \beta_{\lambda_{1}2}=\frac{r^{2}detA}{\beta_{\lambda_{1}1}}.
%\end{equation*}
%Then,
%\begin{equation}\label{}
 %\Phi(x)=z_{0}\xi_{0}+\circ(|x|^{2}),
%\end{equation}
%\begin{equation}\label{}
 %(M_{\lambda_{1}}-\beta_{\lambda_{1}2}I)\xi_{0}=0,
%\end{equation}
%\begin{equation}\label{}
 %(M_{\lambda_{1}}^{*}-\beta_{\lambda_{1}2}I)\xi_{0}^{*}=0,
%\end{equation}
%\begin{equation}\label{}
 %\xi_{0}=(\xi_{01},\xi_{02}), \xi_{0}^{*}=(\xi_{01}^{*},\xi_{02}^{*}),
%\end{equation}
%\begin{equation}\label{}
%z_{0}=z_{0}(x)=\frac{\langle G(x\xi e_{\lambda_{1}}),\xi ^{*}e_{\lambda_{1}}\rangle_{H}}{-\beta_{\lambda_{1}2}
%\langle\xi_{0}e_{\lambda_{1}},\xi_{0}^{*}e_{\lambda_{1}} \rangle_{H}},
%\end{equation}
%and () equivalents to
%\begin{eqnarray}\label{}
%&\frac{dx}{dt}=\beta_{\lambda_{1}1}x+\frac{1}{\langle\xi e_{\lambda_{1}},\xi^{*}e_{\lambda_{1}} \rangle_{H}}
 %\langle G(x\xi e_{\lambda_{1}}+z_{0}(x)\xi_{0}),\xi ^{*}e_{\lambda_{1}}\rangle_{H}\\&=\beta_{\lambda_{1}1}x+Wx^{4}
 %+\circ(|x|^{4}).
%\end{eqnarray}
%The dynamical behavior of the system can be obtained from the reduction system  on above.

%%%%%%%%%%%%%%%%%%%%%%%%%%%%%%%%%%%%%%%%%%%%%%%%%%%%%%%%
%%%%%%%%%%%%%%%%%%%%%%%%%%%%%%%%%%%%%%%%%%%%%%%%%%%%%%%%

\subsubsection{The double real eigenvalues transition of the Schnakenberg system.}

For the PES (\ref{340}), there exists real double eigenvalues transition for $d>d_{0}$. let $i_{2}=i_{1}+1$ and $i=i_{1}$, then the dynamic transition behavior of the system can be dictated by the center manifold reduction equations for this case as follows.
\begin{eqnarray}\label{}
\left\{
   \begin{array}{ll}
   \frac{dx}{dt}=\beta_{\lambda_{i_{1}}}^{2}x+\frac{1}{h_{1}}
    \langle G_{2}(x\xi e_{\lambda_{i_{1}}}+y\eta e_{\lambda_{i_{2}}}),\xi ^{*}e_{\lambda_{i_{1}}}\rangle_{H},
   \\ \frac{dy}{dt}=\beta_{\lambda_{i_{2}}}^{2}y+\frac{1}{h_{2}}
    \langle G_{2}(x\xi e_{\lambda_{i_{1}}}+y\eta e_{\lambda_{i_{2}}}),\eta ^{*}e_{\lambda_{i_{2}}}\rangle_{H},
\end{array}
\right.
\end{eqnarray}
where
\begin{equation}\label{}
  \xi =(\xi_{1} ,\xi_{2} )=(a+b,\frac{(a+b)\lambda_{i_{1}}^{2}-(b-a)r}{r(a+b)^{2}})^T,
\end{equation}
\begin{equation}\label{}
  \eta =(\eta_{1},\eta_{2} )=(a+b,\frac{(a+b)\lambda_{i_{2}}^{2}-(b-a)r}{r(a+b)^{2}})^T,
\end{equation}
\begin{equation}\label{}
  \xi =(\xi_{1}^{*} ,\xi_{2}^{*} )=(\frac{2br(a+b)}{(b-a)r-(a+b)\lambda_{i_{1}}^{2}},a+b)^T,
\end{equation}
\begin{equation}\label{}
  \eta^{*} =(\eta_{1}^{*},\eta_{2}^{*} )=(\frac{2br(a+b)}{(b-a)r-(a+b)\lambda_{i_{2}}^{2}},a+b)^T,
\end{equation}
\begin{equation}\label{}
 h_{1}=\bigg(\frac{2br(a+b)^{2}}{(b-a)r-(a+b)\lambda_{i_{1}}}
  +\frac{(a+b)\lambda_{i_{1}}-(b-a)r}{r(a+b)}\bigg)\int_{\Omega} e_{\lambda_{i_{1}}}^{2}dx,
\end{equation}
\begin{equation}\label{}
 h_{2}=\bigg(\frac{2br(a+b)^{2}}{(b-a)r-(a+b)\lambda_{i_{2}}}
  +\frac{(a+b)\lambda_{i_{2}}-(b-a)r}{r(a+b)}\bigg)\int_{\Omega} e_{\lambda_{i_{2}}}^{2}dx.
\end{equation}
Let
 \begin{align}\label{}
  &s_{1}=-Hr\xi_{1}^{2}+Mr\xi_{1}\xi_{2},\\
   & s_{2}=-2Hr\xi_{1}\eta_{1}+Mr(\xi_{1}\eta_{2}
    +\xi_{2}\eta_{1}),\\
    &s_{3}=-2Hr\eta_{1}^{2}+Mr\eta_{1}\eta_{2},
 \end{align}
where
\begin{equation}\label{}
   H=\frac{b}{(a+b)^{2}},~~~~~M=2(a+b)
\end{equation}
then,
 \begin{align}\label{}
 & G_{2}(x\xi e_{\lambda_{i_{1}}}+y\eta e_{\lambda_{i_{2}}})=(M_{1},-M_{1}),\\
  &M_{1}=s_{1}e_{\lambda_{i_{1}}}^{2}x^{2}
  +s_{2}e_{\lambda_{i_{1}}}e_{\lambda_{i_{2}}}xy
  +s_{3}e_{\lambda_{i_{2}}}^{2}y^{2}.
 \end{align}

Denote
 \begin{align}\label{}
 &a_{20}=\frac{s_{1}(\xi_{1}^{*}-\xi_{2}^{*})}{h_{1}}\int_{\Omega} e_{\lambda_{i_{1}}}^{3}dx,\\
 &a_{11}=\frac{s_{2}(\xi_{1}^{*}-\xi_{2}^{*})}{h_{1}}\int_{\Omega} e_{\lambda_{i_{1}}}^{2}e_{\lambda_{i_{2}}}dx,\\
 &a_{02}=\frac{s_{3}(\xi_{1}^{*}-\xi_{2}^{*})}{h_{1}}\int_{\Omega} e_{\lambda_{i_{2}}}^{2}e_{\lambda_{i_{1}}}dx,\\
 &b_{20}=\frac{s_{1}(\eta_{1}^{*}-\eta_{2}^{*})}{h_{2}}\int_{\Omega} e_{\lambda_{i_{2}}}^{3}dx,\\
 &b_{11}=\frac{s_{2}(\eta_{1}^{*}-\eta_{2}^{*})}{h_{2}}\int_{\Omega} e_{\lambda_{i_{2}}}^{2}e_{\lambda_{i_{1}}}dx,\\
 &b_{02}=\frac{s_{3}(\eta_{1}^{*}-\eta_{2}^{*})}{h_{2}}\int_{\Omega} e_{\lambda_{i_{1}}}^{2}e_{\lambda_{i_{2}}}dx,
 \end{align}
then, (\ref{eq:sch}) can be written as
\begin{eqnarray}\label{376}
\left\{
   \begin{array}{ll}
   \frac{dx}{dt}=\beta_{\lambda_{i_{1}}}^{2}x
   +a_{20}x^{2}+a_{11}xy+a_{02}y^{2},
   \\ \frac{dy}{dt}=\beta_{\lambda_{i_{2}}}^{2}y
   +b_{20}x^{2}+b_{11}xy+b_{02}y^{2}.
\end{array}
\right.
\end{eqnarray}

Denote by
\begin{align}\label{}
  &A_{1}=\beta_{\lambda_{i_{1}}1}+2a_{20}x_{0}+a_{11}y_{0},
  \\
  &A_{2}=2y_{0}a_{02}+a_{11}y_{0},
  \\
  &A_{3}=2b_{02}x_{0}^{2}+b_{11}y_{0},
  \\
  &A_{4}=\beta_{\lambda_{i_{2}}1}+2b_{20}y_{0}+b_{11}x_{0}.
\end{align}

Hence, the transition for (\ref{eq:sch}) is equivalent to the transition for (\ref{376}). The transition theorem is stated as follows.
\begin{thm}\label{thm:sch_2}
The system (\ref{eq:sch}) has a stable steady state $(0,0)$ for $d<d_{0}$, but bifurcates a new steady state $(x_{0},y_{0})$  for $d>d_{0}$, if and only if the following conditions hold true
\begin{eqnarray}\label{382}
  A_{1}+A_{4}<0,
  A_{1}A_{4}-A_{2}A_{3}>0,
\end{eqnarray}
and there is a stable attractor $u^{d}$ of the system bifurcated from $d>d_{0}$  as follows
\begin{equation*}\label{}
  u^{d}=x_{0}\xi e_{\lambda_{i_{1}}}+y_{0}\eta e_{\lambda_{i_{2}}}
\end{equation*}
\end{thm}
\begin{proof}
The following matrix is the Jacobian of the system at $(x_{0},y_{0})$
\begin{equation}
\left(
\begin{array}{cc}
    A_{1}& A_{2}\\
    A_{3}& A_{4}
    \end{array}
\right)
\end{equation}
The condition (\ref{382}) tell us that the $(x_{0},y_{0})$ is a stable attractor, which can be written as $u^{d}=x_{0}\xi e_{\lambda_{i_{1}}}+y_{0}\eta e_{\lambda_{i_{2}}} $.
\end{proof}

%%%%%%%%%%%%%%%%%%%%%%%%%%%%%%%%%%%%%%%%%%%%%%%%%%%%%%%%
%%%%%%%%%%%%%%%%%%%%%%%%%%%%%%%%%%%%%%%%%%%%%%%%%%%%%%%%

\subsection{An illustrative example}
Taking $\Omega$ to be the rectangle $[0,10]\times[0,5]$, if we consider Dirichlet boundary condition, then we have
  \begin{eqnarray}
   e_{m,n}(x)=cos(\frac{m\pi x}{10})cos(\frac{n\pi y}{5}),
    \\ \lambda_{m,n}(x)=(\frac{m\pi}{10})^{2} + (\frac{n\pi}{5})^{2},
 \end{eqnarray}
and $m,n=0,1,2,3,4\ldots$.

By directly calculating, we can get
\begin{align}\label{eq:sch_ex_lambda}
\begin{aligned}
  &\lambda_{0,1}=0.3948, \lambda_{1,0}=0.0987, \\
  &\lambda_{2,0}=0.3948, \lambda_{1,1}=0.4935 ,\\
  &\lambda_{1,2}=1.6778, \lambda_{0,2}=1.5791,
  \end{aligned}
\end{align}
\begin{figure}[!hbtp]
  \centering
  \includegraphics[width=.8\textwidth]{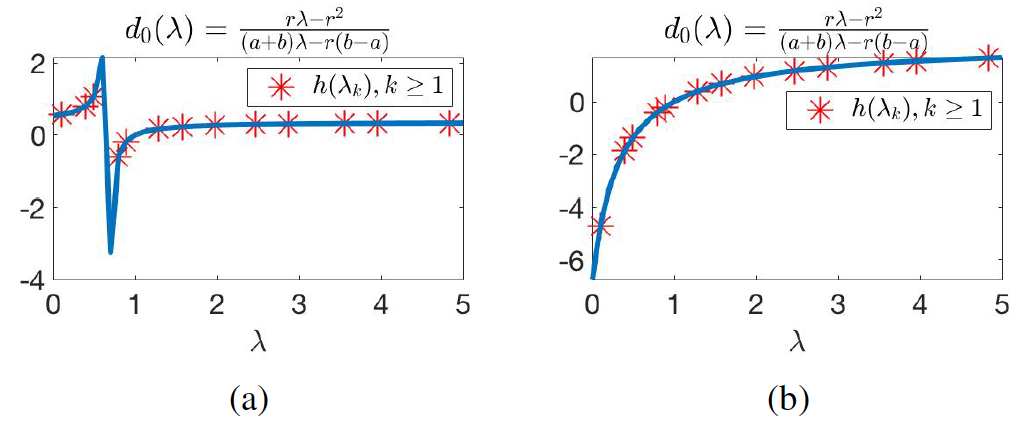}
  \caption{Plot of the critical parameter $d_0^{\lambda}$ as a function of $\lambda$ when (a) $a=0.1,b=0.5,r=1$, (b) $a=1,b=0.5,r=1$. The red stars correspond to the eigenvalues of $-\Delta$ in $\Omega$ as calculated in \eqref{eq:sch_ex_lambda}}\label{fig:sch_d0}
\end{figure}

For system (\ref{eq:sch}), we consider two sets of parameters differing only in $a$, specifically
\bea
&\text{~(a)~} &a=0.1, b=0.5,r=1, \label{eq:example_coeff_a} \\
&\text{~(b)~} &a=1, b=0.5,r=1. \label{eq:example_coeff_b}
\eea
then for case (a), we get  (calculated using \textit{MATLAB})
\begin{eqnarray}
D_{v}^{0}=5.1648, k_{c}^{0}=0.2985
\end{eqnarray}
since $\lambda_{i}$ in Theorem \ref{thm:main_cont} is $\lambda_{1,0}=0.0987$ and $\lambda_{0,1}=0.3948 $, for case (a),  $\lambda_{1,0}< k_{c}^{0} < \lambda_{0,1}$ is satisfied and the critical parameter $d_{0} = $ is given by
\bea
d_{0}=\min(d_{0}^{\lambda_{1,0}},d_{0}^{\lambda_{0,1}})= 0.6157.
\eea
Therefore, Turing's instability and the transition of the system (\ref{eq:sch}) with coefficients  (\ref{eq:example_coeff_a}) occurs when diffusion rate of $v$ is greater than $d_{0}=0.6157$. The numerical simulation of this case is illustrated in Fig. \ref{fig:sch_unstable} below.

 \begin{figure}[!ht]
 	\centering
	\includegraphics[width = .8\textwidth]{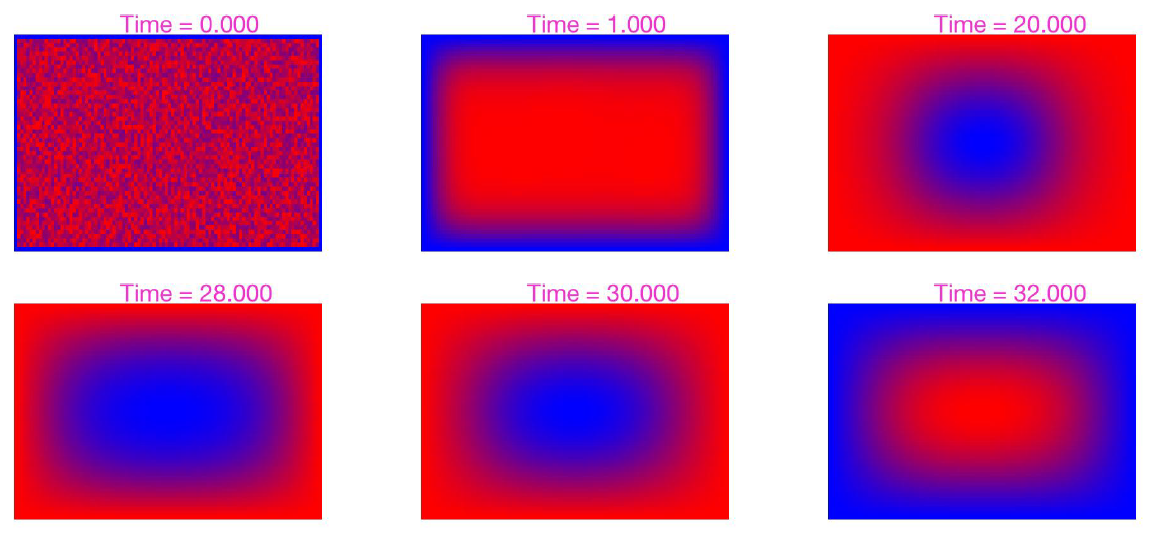}
 	\caption{Numerical simulation of the Schnakenberg system for case (a) when $d=1$ under Dirichlet boundary condition $(u,v) = (u_0,v_0)$ on $\partial \Omega$. The images shows temporal snapshots of spatial distributions of chemical $V$. In this case, the homogeneous steady state losses its stability and generates dynamics patterns (periodic solutions).}\label{fig:sch_unstable}
 \end{figure}

For case (b), however, we get
\begin{eqnarray}
D_{v}^{0}=55.1025, k_{c}^{0}=-0.1871
\end{eqnarray}
since $k_{c}^{0} < 0 < \inf_{k\ge0}{(\lambda_k)}$, the assumption \eqref{eq:sch_ass} does not apply. However, because $D=-2.5833, L = -0.5816<0$, the system must be stable. The numerical simulation of this case is illustrated in Fig. \ref{fig:sch_stable} below.
 \begin{figure}[!ht]
 	\centering
	\includegraphics[width = .8\textwidth]{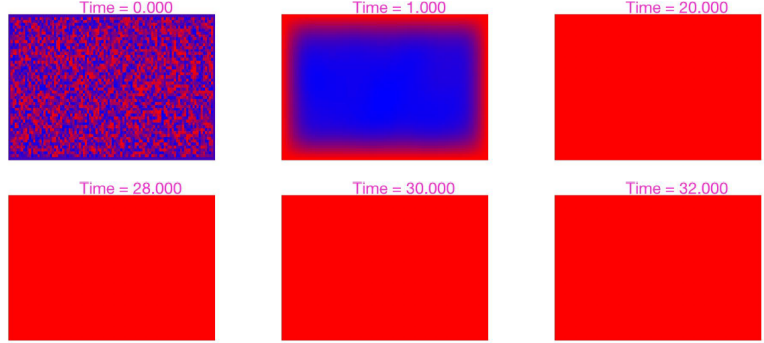}
 	\caption{Numerical simulation of the Schnakenberg system for case (b) when $d=1$ under Dirichlet boundary condition $(u,v) = (u_0,v_0)$ on $\partial \Omega$ and was purturbed. The images shows temporal snapshots of spatial distributions of chemical $V$. In this case, no Turing's instability occurs.}
	\label{fig:sch_stable}
 \end{figure}

%\subsection{An illustrative example}
% Taking $\Omega$ is a rectangle, and the ratio of length and width is 1:2. Let $r=100 $ , then we have
%  \begin{eqnarray}
%   e_{n,m}(x)=cos(\frac{n\pi x}{2})sin(m\pi y),
%    \\ \lambda_{n,m}(x)=(\frac{n\pi}{2})^{2}+(m\pi)^{2},
% \end{eqnarray}
%and $n,m=0,1,2,3,4\ldots$.
%
%For system (\ref{eq:sch}), we take
%\begin{equation}\label{}
%a=0.5, b=1.5,r=100.
%\end{equation}
%By directly calculating, we can get
%\begin{align}\label{385}
%  &\lambda_{0,1}=9.8596, \lambda_{1,0}=2.4694, \\
%  &\lambda_{2,0}=9.8596, \lambda_{1,1}=12.3245 ,\\
%  &\lambda_{1,2}=41.9033, \lambda_{0,2}=39.4384,\\
%  &k_{c}=\frac{r(b-a)d-r(a+b)^{3}}{2d(a+b)}
%\end{align}
%Then we get
%\begin{eqnarray}
%D_{v}^{0}=47.8, k_{c}^{0}=20.8
%\end{eqnarray}
%
%Then we obtained that $\lambda_{i}$ in Theorem 2.4 is $\lambda_{1,1}=12.3245$ and $\lambda_{0,2}=39.4384 $ and the critical point $d_{0}$ is given by
%\begin{align}\label{}
%&d_{0}^{\lambda_{1,1}}=\frac{-r^{2}(a+b)^{3}-r(a+b)^{3}\lambda_{1,1}}{(a+b)\lambda_{1,1}-r(b-a)}=443.48,\\
%&d_{0}^{\lambda_{0,2}}=\frac{-r^{2}(a+b)^{3}-r(a+b)^{3}\lambda_{0,2}}{(a+b)\lambda_{0,2}-r(b-a)}=370.11,\\
%&d_{0}=\min\{ d_{0}^{\lambda_{1,1}}, d_{0}^{\lambda_{0,2}}\}=370.11.
%\end{align}
%And the Turing's instability and the transition of the system (\ref{eq:sch}) with coefficient (\ref{385}) occurs at $d_{0}=370.11$.
%\newpage

%%%%%%%%%%%%%%%%%%%%%%%%%%%%%%%%%%%%%%%%%%%%%%%%%%%%%%%%
%%%%%%%%%%%%%%%%%%%%%%%%%%%%%%%%%%%%%%%%%%%%%%%%%%%%%%%%

\section{Discussion}
In this paper, we have studied the stability and dynamic transition property of Turing's systems, using the new dynamical transition theory developed in \cite{Ma2005,Ma2015}. Turing's systems are famous for the so-called \textit{diffusion-driven instabilities}, resulting emerging patterns as a result of bifurcations from the uniform steady state. Our analysis showed that only two bifurcation parameters are enough to dictate all the dynamics of the Turing's system near a steady state, as shown in Section \ref{sec:turing}. Besides, with the help of center manifold reduction procedure, we are able to \textit{locally} derive all the dynamical behavior of a infinite dimensional system near a steady state. 

Turing's instability is a widely studied topic in nonlinear science, lots of works had been done with respect to this. Most related works, however, mainly address on analyzing asymptotic behavior or numerical simulations, and mostly focus on a specific Turing's system, e.g. the Brusselator model, the Gray-Scott system, the Schnakenberg system etc., as introduced in Section \ref{sec:intro}. Our work, on the other hand, started from the general model \eqref{eq:model} and its analytical properties. From these analytical properties, by calculating critical parameters and bifurcated solutions, we found that Turing's instability can be recognized as a \textit{critical phenomena}, and consequently can be studied in a similar fashion as other physical phenomena like superconductivity, crystallization etc \cite{MaBook2015MPTP,Ma2015}. Dynamic transition theory is not only a fundamental method of studying critical phenomena, but also a powerful tool of studying dynamics of a PDE system.

The method introduced in this paper is systematic, and can be applied to various different kinds of two-component reaction diffusion systems (Turing's systems). In the future, we propose to build up a general method to study more complicated phase transition behavior of a reaction diffusion system, for example, second order phase transition and its biological and chemical implications. In general, chaotic behavior will take place for a physical system underwent second and higher order phase transitions under modern classification \cite{Kosterlitz1973,Sole1996,Ojovan2013,Ma2015}, and fruitful critical phenomena structures are found in numerous experiments for each order of transition. In the case of finite dimension, this kind of problems has been well studied from mathematical point of view \cite{WigginsBook2003,GuckenheimerBook2013}. Since most physical, chemical and biological phenomena are modeled by PDEs, accompanying with infinite dimension operators, constructing a systematic theory of studying second and higher order phase transition on infinite dimensional dynamical systems is of great importance for linking experimental results and mathematical theory, and that will be our next goal.

\bibliographystyle{plain}
\bibliography{Turing}

\iffalse
\begin{center}

\end{center}
\fi
\end{document}